\documentclass[11pt,reqno]{amsart}
\usepackage{amssymb,amsmath,amsthm,newlfont}
\usepackage[shortlabels]{enumitem}

\theoremstyle{plain}
\newtheorem{theorem}{Theorem}[section]
\newtheorem{proposition}[theorem]{Proposition}
\newtheorem{lemma}[theorem]{Lemma}
\newtheorem{corollary}[theorem]{Corollary}
\newtheorem{problem}[theorem]{Problem}

\theoremstyle{definition}

\theoremstyle{remark}
\newtheorem*{remark}{Remark}

\newcommand{\cB}{\mathcal{B}}
\newcommand{\cD}{\mathcal{D}}
\newcommand{\cE}{\mathcal{E}}
\newcommand{\cH}{\mathcal{H}}

\newcommand{\CC}{\mathbb{C}}
\newcommand{\DD}{\mathbb{D}}
\newcommand{\RR}{\mathbb{R}}
\newcommand{\ZZ}{\mathbb{Z}}
\newcommand{\TT}{\mathbb{T}}

\renewcommand{\hat}{\widehat}

\DeclareMathOperator{\supp}{supp}

\begin{document}

\title{A pathological de Branges--Rovnyak space}

\author[T. Ransford]{Thomas Ransford}
\address{D\'epartement de math\'ematiques et de statistique, Universit\'e Laval, Qu\'ebec (QC), G1V 0A6, Canada}
\email{ransford@mat.ulaval.ca}

\date{10 September 2026}

\begin{abstract}
We construct a de Branges--Rovnyak space containing the polynomials as a dense subspace, together with a function in the space 
that lies outside the closed linear span of its Taylor partial sums.
The space has the further property that the monomials are bounded in norm. The construction is carried out by reducing it to an elementary problem about series, which
we then  solve.
\end{abstract}

\thanks{Research supported  by NSERC Discovery Grant RGPIN-2026-04565}

\keywords{Function spaces, Taylor series, Summability methods, De Branges--Rovnyak spaces}

\makeatletter
\@namedef{subjclassname@2020}{\textup{2020} Mathematics Subject Classification}
\makeatother

\subjclass[2020]{primary 30H45; secondary 40J05, 41A58}

\maketitle

\section{Introduction}\label{S:Intro}

Let $X$ be a \emph{holomorphic function space} on the open unit disk $\DD$.
By this, we mean that $X$ is a Banach space 
of holomorphic functions on $\DD$, containing the polynomials, 
and with the property that convergence in the norm of $X$ implies
local uniform convergence in $\DD$. 

For $f\in X$ and $n\ge0$, we write $s_n(f)(z):=\sum_{k=0}^n \hat{f}(k)z^k$,
where $\hat{f}(k):=f^{(k)}(0)/k!$.
Note that $s_n(f)\in X$, as $X$  contains the polynomials.
Of course $s_n(f)\to f$ locally uniformly in $\DD$ as $n\to\infty$. 
What about convergence in the norm of $X$?

Various possibilities can arise:

\begin{enumerate}[(I)]
\item \label{i:1}
For all $f\in X$, we have $\|s_n(f)-f\|_X\to0$ as $n\to\infty$.
Many spaces exhibit this `good' behaviour, for example the Hardy spaces $H^p$ for $1<p<\infty$, 
the Bergman spaces $A^p$ for $1<p<\infty$, and the Dirichlet space $\cD$.
\item \label{i:2}
It is not the case that $\|s_n(f)-f\|_X\to0$ for all $f\in X$, but $X$ \emph{admits a convergent summability method}.
This means that there exists an infinite complex matrix $A=(a_{nk})$ such that, for each $f\in X$ and each $n\ge0$,
the series $S_n^A(f):=\sum_{k=0}^\infty a_{nk}s_k(f)$ converges in the norm of~$X$ and
$\|S_n^A(f)-f\|_X\to0$ as $n\to\infty$.
For example, this is the case if $X$ is the disk algebra $A(\DD)$, the Hardy space $H^1$,
the Bergman space $A^1$, the space $VMOA$ of analytic functions of vanishing mean oscillation,
the little Bloch space $\cB_0$, and the local Dirichlet space~$\cD_1$.

\item \label{i:3}
Polynomials are not norm-dense in $X$. Evidently, in this case, $X$  admits no convergent 
summability method. Examples of spaces of this type include the Hardy space $H^\infty$,
the Bloch space $\cB$, and the space $BMOA$ of analytic functions of bounded mean oscillation.
\end{enumerate} 

For details on all these cases, we refer to the article \cite{GMR24} and the references cited therein.

In this paper, we study the case of de Branges--Rovnyak spaces
$\cH(b)$ (these will be defined in detail in \S\ref{S:deBR}).
Our main result exhibits an example of 
a de Branges--Rovnyak space that
is pathological, in the sense that it
falls into none of the categories \ref{i:1}--\ref{i:3} above.

\begin{theorem}\label{T:main}
There exist a de Branges--Rovnyak space $\cH(b)$ and a function
$f\in\cH(b)$ with the following properties:
\begin{itemize}
\item $\cH(b)$ contains the polynomials as a dense subspace;
\item $f$ lies outside the closed subspace spanned by $\{s_n(f):n\ge0\}$;
\item $\sup_{n\ge0}\|z^n\|_{\cH(b)}<\infty$;
\item $b$ has real coefficients.
\end{itemize}
\end{theorem}

This theorem solves a number of previously open problems.

First of all, it answers a question posed in \cite{MR18}
as to whether every de Branges--Rovnyak space
in which polynomials are dense admits a convergent
summability method. It was previously known that
various specific summability methods can fail
\cite{EFKMR16, GMR24, MPR22b, MR18}, but the general question remained open.
Theorem~\ref{T:main} now provides a definitive negative answer.

Secondly, it is known
that a necessary condition
for $\cH(b)$ to fall into category~\ref{i:1} above
is that $\sup_{n\ge0}\|z^n\|_{\cH(b)}<\infty$
(see \cite[Corollary~6.11]{GMR24}).
Theorem~\ref{T:main} shows that the converse fails
in the following strong sense: 
the condition $\sup_{n\ge0}\|z^n\|_{\cH(b)}<\infty$
is insufficient to imply that $\cH(b)$ falls into 
any of the  categories~\ref{i:1}--\ref{i:3}.

Thirdly, we believe that the space $\cH(b)$ is the
first known example of a `concrete' function space that lies outside
categories~\ref{i:1}--\ref{i:3}.
The previous examples of pathological function spaces,  
presented in  \cite{MPR22a} and \cite{MR19}, were 
arti\-ficial constructs, not naturally occurring spaces.
Moreover, in all of these examples, the monomials $z^n$
were unbounded. Thus the  $\cH(b)$ in Theorem~\ref{T:main}
is  the first instance of a function space,
either natural or artificial, 
that lies outside categories~\ref{i:1}--\ref{i:3},
but in which $\sup_{n\ge0}\|z^n\|<\infty$.

We conclude this section with an application
of Theorem~\ref{T:main} to dilations.
Given $f$ holomorphic on $\DD$ and $\omega\in\DD$,
the \emph{$\omega$-dilation} of $f$ is the holomorphic function $f_\omega$ defined by $f_\omega(z):=f(\omega z)~(z\in\DD)$.

\begin{corollary}\label{C:dilate}
Let $\cH(b)$ and $f$ be as in Theorem~\ref{T:main}.
Then $f_\omega\in\cH(b)$ for all $\omega\in\DD$,
and $f$ lies outside the closed subspace spanned
by $\{f_\omega:\omega\in\DD\}$.
\end{corollary}

\begin{proof}
Both parts of the corollary will follow if we can show that,
for each $\omega\in\DD$,
the function $f_\omega$ belongs to the closed subspace spanned
by $\{s_n(f):n\ge0\}$. 

Fix $\omega\in\DD$. Then it is elementary
that, as functions on $\DD$,
\begin{equation}\label{E:fomega}
f_\omega=\sum_{k\ge0}(1-\omega)\omega^ks_k(f).
\end{equation}
Moreover, the right-hand side of \eqref{E:fomega} converges
absolutely in the norm of $\cH(b)$. Indeed, since $\sup_{n\ge0}\|z^n\|_{\cH(b)}<\infty$, it follows that $\|s_k(f)\|_{\cH(b)}=O((1+\epsilon)^k)$ as $k\to\infty$ for each $\epsilon>0$, and this implies that 
$\sum_{k\ge0}|1-\omega||\omega|^k\|s_k(f)\|_{\cH(b)}<\infty$.
Hence $f_\omega$ lies in the closed subspace spanned
by $\{s_n(f):n\ge0\}$, as desired.
\end{proof}

The rest of the paper is organized as follows.
The necessary  background on general Hilbert function spaces
and on de Branges--Rovnyak spaces
is presented in \S\ref{S:Hilbert} 
and  \S\ref{S:deBR} respectively.
In \S\ref{S:equiv}, we reduce Theorem~\ref{T:main} 
 to an apparently elementary problem about series.
The rest of the article is devoted to  solving 
this problem, thereby completing the proof of Theorem~\ref{T:main}.


\section{Hilbert holomorphic function spaces}\label{S:Hilbert}

Let $H$ be a Hilbert holomorphic function space on $\DD$. 
Recall that this means that $H$ contains the polynomials, and that
convergence in norm implies local uniform convergence in $\DD$.
It follows that, for each $k\ge0$, the map $f\mapsto \hat{f}(k):H\to\CC$ is a
continuous linear functional on $H$.
Therefore, for each $n\ge0$, the map 
$s_n:f\mapsto s_n(f):H\to H$ is a bounded linear projection. 
We denote its operator norm by $\|s_n\|$.
Also, given $f\in H$, we write $\supp\hat{f}:=\{k\ge0:\hat{f}(k)\ne0\}$.

\begin{theorem}\label{T:Hilbert}
Let $H$ be a Hilbert holomorphic function space on $\DD$.
Then the following are equivalent:
\begin{enumerate}[\normalfont(i)]
\item there exists $f\in H$ such that $f$ lies outside the closed linear span of $\{s_n(f):n\ge0\}$;
\item there exists $g\in H\setminus\{0\}$  such that  $\langle g,z^n\rangle_H=0$ for all $n\in\supp\hat{g}$.
\end{enumerate}
If there exists $g$ as in (ii), then $\supp\hat{g}$ must be infinite.
If, further, polynomials are dense in $H$, then $\ZZ^+\setminus \supp\hat{g}$ is also infinite.
\end{theorem}

\begin{proof}
(ii)$\Rightarrow$(i): Suppose that (ii) holds.
Then $\langle g,s_n(g)\rangle_H=0$ for all $n$, and hence
$\langle g,h\rangle_H=0$ for all $h$ in the closed linear span of $\{s_n(g):n\ge0\}$. 
Since $g\ne0$, we have $\langle g,g\rangle_H\ne0$, and so $g$ lies outside this closed linear span.
Thus (i) holds with $f=g$.

(i)$\Rightarrow$(ii): Suppose that (i) holds. Set $S:=\supp\hat{f}$ and
$M:=\{h\in H: \supp \hat{h}\subset S\}$.
Then $M$ is a closed subspace of $H$, hence a Hilbert space in its own right.
Obviously $f\in M$ and also $s_n(f)\in M$ for all $n\ge0$. Since $s_0(f)=\hat{f}(0)$ and $s_n(f)-s_{n-1}(f)=\hat{f}(n)z^n$ for all $n\ge1$,
it follows that the span of $\{s_n(f):n\ge0\}$ equals the span of $\{z^n:n\in S\}$.
As $f$ is assumed to lie outside the closed linear span of $\{s_n(f):n\ge0\}$,
it also lies outside the closed linear span of $\{z^n:n\in S\}$. Therefore 
there exists $g\in M\setminus\{0\}$ such that $g\perp z^n$ for all $n\in S$.
Thus (ii) holds.

Suppose, if possible, that (ii) holds with $\supp\hat{g}$ finite. Then
\[
\langle g,g\rangle_H
=\Bigl\langle \sum_{n\in \supp\hat{g}}\hat{g}(n)z^n,g\Bigr\rangle_H
=\sum_{n\in \supp\hat{g}}\hat{g}(n)\langle z^n, g\rangle_H=0,
\]
which is a contradiction. Therefore $\supp\hat{g}$ must be infinite.

Finally, assume that polynomials are dense in $H$, and suppose,
 if possible, that (ii) holds with $\ZZ^+\setminus \supp\hat{g}$ finite.
Fix  a positive integer $N$, chosen large enough so that 
$\ZZ^+\setminus \supp\hat{g}\subset\{0,1,\dots, N\}$. 
Let $\epsilon>0$. By density, there exists a polynomial $p$ such that $\|g-p\|_H<\epsilon$. 
Set $q:=p+s_N(g-p)$. Then $\hat{q}(n)=0$ for all $n\in \ZZ^+\setminus \supp\hat{g}$,
so $\langle q,g\rangle_H=0$. It follows that
\begin{align*}
|\langle g,g\rangle_H|
&=|\langle g-q,g\rangle_H|\\
&=|\langle g-p-s_N(g-p),\,g\rangle_H|\\
&\le (1+\|s_N\|)\|g-p\|_H\|g\|_H\\
&\le (1+\|s_N\|)\epsilon\|g\|_H.
\end{align*}
As $\epsilon>0$ is arbitrary, 
it follows that $\langle g,g\rangle_H=0$, a contradiction. 
We conclude that $\ZZ^+\setminus \supp\hat{g}$ must be infinite.
\end{proof}

\begin{remark}
It is not altogether obvious whether the situation described in Theorem~\ref{T:Hilbert}\,(ii) can ever arise. 
However, it was shown in \cite[Theorem~3.1]{MPR22a} that, for each set $S\subset\ZZ^+$ such that both $S$ and $\ZZ^+\setminus S$
are infinite, there exist a Hilbert holomorphic function space $H$ containing the polynomials as a dense subspace and a
function $g\in H\setminus\{0\}$ such that $\supp\hat{g}\subset S$ and $\langle g,z^n\rangle_H=0$ for all $n\in S$.
\end{remark}


\section{De Branges--Rovnyak spaces}\label{S:deBR}

The spaces now called de Branges--Rovnyak spaces 
were introduced by de Branges and Rovnyak 
in the appendix of \cite{dBR66a} 
and further studied in \cite{dBR66b}.
They were later popularized by the book of Sarason \cite{Sa94}.
It is now known that these spaces have a variety of connections 
with other topics in complex analysis and operator theory,
and they have attracted the interest of many authors.
In this section we present  a brief review of these spaces.
The interested reader may find much more information in the books
of Sarason \cite{Sa94} and Fricain--Mashreghi \cite{FM16}.

Let $b\in H^\infty$ with $\|b\|_{H^\infty}\le 1$. 
The map $(z,w)\mapsto (1-\overline{b(w)}b(z))/(1-\overline{w}z)$ is
positive definite on $\DD\times\DD$, so it is the reproducing kernel 
of a unique Hilbert  space of functions on $\DD$. This space is
called the \emph{de Branges--Rovnyak space} associated to $b$,
and is denoted by $\cH(b)$. 

Each $\cH(b)$-space is  contractively contained
in the Hardy space $H^2$. 
In particular, convergence in the norm of $\cH(b)$
implies local uniform convergence in $\DD$.
This is one of the two qualifying conditions to be 
a holomorphic function space on $\DD$.
The other, namely that $\cH(b)$ contain the polynomials, 
is only satisfied for certain $b$. This is the subject of the 
following result. We write $\TT$ for the unit circle.

\begin{theorem}[\protect{\cite[\S\S IV-2, IV-3, NIV-1, V-I]{Sa94}}]\label{T:non-extreme}
Let $b\in H^\infty$ with $\|b\|_{H^\infty}\le 1$.
Then the following statements are equivalent:
\begin{enumerate}[\normalfont(i)]
\item $\cH(b)$ contains the polynomials;
\item $b$ is a non-extreme point of the unit ball of $H^\infty$;
\item $\log(1-|b|^2)\in L^1(\TT)$.
\end{enumerate}
Moreover, if (i)--(iii) hold, then
the polynomials are dense in $\cH(b)$.
\end{theorem}

Using the equivalence between (ii) and (iii) in Theorem~\ref{T:non-extreme},
we see that every non-extreme $b$ has a unique \emph{Pythagorean mate},
namely an outer function $a$ on $\DD$ such that $a(0)>0$
and $|b|^2+|a|^2=1$ a.e.\ on~$\TT$. 
Conversely, if $b$ has a Pythagorean mate $a$,
then it is non-extreme.

The next result characterizes those $\cH(b)$ for which $\sup_{n\ge0}\|z^n\|_{\cH(b)}<\infty$.

\begin{theorem}[\protect{\cite[Theorem~1]{Sa86}}]\label{T:tame}
Let $b$ be a non-extreme point of the unit ball of $H^\infty$, let $a$ be its Pythagorean mate
and let $\phi:=b/a$. Then the following statements are equivalent.
\begin{enumerate}[\normalfont\rm(i)]
\item $\sup_{n\ge0}\|z^n\|_{\cH(b)}<\infty$;
\item $\phi\in H^2$;
\item $1/(1-|b|^2)\in L^1(\TT)$.
\end{enumerate}
\end{theorem}

The next result is a formula for  $\cH(b)$-norms in terms of Taylor coefficients.

\begin{theorem}[\protect{\cite[Theorem~3.1]{Ra26}}]\label{T:formula}
Let $b$ be a non-extreme point of the unit ball of $H^\infty$, let $a$ be its Pythagorean mate
and let $\phi:=b/a$. Assume that $\phi\in H^2$.
Then $f\in\cH(b)$ if and only if $f\in H^2$ and
\[
\sum_{k\ge0}\Bigl|\sum_{j\ge0}\overline{\hat{\phi}(j)}\hat{f}(j+k)\Bigr|^2<\infty,
\]
in which case
\[
\|f\|_{\cH(b)}^2=\sum_{k\ge0}|\hat{f}(k)|^2
+\sum_{k\ge0}\Bigl|\sum_{j\ge0}\overline{\hat{\phi}(j)}\hat{f}(j+k)\Bigr|^2.
\]
\end{theorem}

Using the polarization identity, we deduce a formula for inner products in $\cH(b)$.

\begin{corollary}\label{C:formula}
Under the hypotheses of Theorem~\ref{T:formula},
if $f,g\in\cH(b)$, then
\[
\langle f,g\rangle_{\cH(b)}=\sum_{k\ge0}\hat{f}(k)\overline{\hat{g}(k)}
+\sum_{k\ge0}\Bigl(\sum_{i\ge0}\overline{\hat{\phi}(i)}\hat{f}(i+k)\Bigr)\Bigl(\sum_{j\ge0}{\hat{\phi}(j)}\overline{\hat{g}(j+k)}\Bigr).
\]
\end{corollary}

We conclude by remarking that Theorem~\ref{T:formula} and Corollary~\ref{C:formula} may break down without the assumption that $\phi\in H^2$.
This phenomenon is discussed in detail in \cite{Ra26}.


\section{Reduction to an elementary problem}\label{S:equiv}

The following result reduces the task of proving
Theorem~\ref{T:main} to that of solving an apparently
elementary problem about series.
In what follows,  
given a complex sequence $\alpha=(\alpha_j)_{j\ge0}$, 
we write
$\supp\alpha:=\{j\ge0:\alpha_j\ne0\}$.

\begin{theorem}\label{T:equiv}
The following statements are equivalent.
\begin{enumerate}[\normalfont(i)]
\item There exist a de Branges--Rovnyak space $\cH(b)$ and a
function $f\in\cH(b)$ satisfying the conclusions of
Theorem~\ref{T:main}.
\item There exist sequences $\alpha,\beta\in\ell^2(\ZZ^+,\CC)$ and $\gamma\in\ell^2(\ZZ^+,\RR)$, none of them identically zero,  such that
\begin{align}
\label{E:cond1}\sum_{j\ge k}\gamma_{j-k}\alpha_{j}&=\beta_k \quad(k\in\ZZ^+),\\ 
\label{E:cond2}\sum_{0\le k\le j}\gamma_{j-k}\beta_k&=-\alpha_j  \quad(j\in \supp\alpha).
\end{align}
\end{enumerate}
\end{theorem}

\begin{proof}[Proof that (i)$\Rightarrow$(ii).]\let\qed\relax
Suppose that $\cH(b)$ and $f$ exist as in (i). 
By  Theorem~\ref{T:non-extreme},
the fact that $\cH(b)$ contains the polynomials implies that $b$ is non-extreme, 
so has a Pythagorean mate~$a$.
By Theorem~\ref{T:tame}, the fact that
$\sup_{n\ge0}\|z^n\|_{\cH(b)}<\infty$ implies that $\phi:=b/a\in H^2$. 
Therefore Theorem~\ref{T:formula} and Corollary~\ref{C:formula} both hold. Also, since $b$ has real coefficients, the same is
true of $a$ and $\phi$.

Now we apply Theorem~\ref{T:Hilbert}:
since $f$ lies outside the closed linear span
of $\{s_n(f):n\ge0\}$,    there exists $g\in\cH(b)\setminus\{0\}$  such that 
\begin{equation}\label{E:conds}
\langle g,z^j\rangle_{\cH(b)}=0 \quad(j\in \supp\hat{g}).
\end{equation}
As $g\in\cH(b)$, Theorem~\ref{T:formula} implies that 
\[
\sum_{k\ge0}\Bigl|\sum_{j\ge0}\overline{\hat{\phi}(j)}\hat{g}(j+k)\Bigr|^2<\infty.
\]
For $k\in\ZZ^+$, set
\[
\alpha_k:=\hat{g}(k), 
\qquad
\beta_k:=\sum_{j\ge0}\overline{\hat{\phi}(j)}\hat{g}(j+k)
\quad\text{and}\quad
\gamma_k:=\hat{\phi}(k).
\]
Then $\alpha,\beta\in\ell^2(\ZZ^+,\CC)$ and $\gamma\in\ell^2(\ZZ^+,\RR)$. From the definition of $\beta$, we have
\[
\beta_k=\sum_{j\ge0}\gamma_j\alpha_{j+k} \quad(k\in\ZZ^+),
\] 
which is equivalent to \eqref{E:cond1}. Also,  by Corollary~\ref{C:formula}, for all $j\in\ZZ^+$ we have
\[
\langle g,z^j\rangle_{\cH(b)}=\hat{g}(j)+\sum_{0\le k\le j}\Bigl(\sum_{i\ge0}\overline{\hat{\phi}(i)}\hat{g}(i+k)\Bigr)\hat{\phi}(j-k)=\alpha_j+\sum_{0\le k\le j}\beta_k\gamma_{j-k}.
\]
Combining this with \eqref{E:conds}, we deduce that
\[
\alpha_j+\sum_{0\le k\le j}\beta_k\gamma_{j-k}=0 \quad(j\in\supp\alpha),
\]
which is equivalent to \eqref{E:cond2}.
Finally, since $g\ne0$, we have $\alpha\not\equiv0$, and from \eqref{E:cond2} it follows that
$\beta\not\equiv0$ and $\gamma\not\equiv0$ as well.
\end{proof}

\begin{proof}[Proof that (ii)$\Rightarrow$(i).] 
Suppose now that there exist $\alpha,\beta\in\ell^2(\ZZ^+,\CC)$ and $\gamma\in \ell^2(\ZZ^+,\RR)$ satisfying
\eqref{E:cond1} and \eqref{E:cond2},
none of them identically zero.
Define $\phi$ by
\[
\phi(z):=\sum_{k\ge0}\gamma_k z^k.
\]
Since $\gamma\in\ell^2(\ZZ^+)$,
we have $\phi\in H^2$. 
In particular $\log(1+|\phi|^2)\in L^1(\TT)$, 
so there exists a unique outer function $a$ on $\DD$ 
such that $a(0)>0$ and $|a|^2=1/(1+|\phi|^2)$ a.e.\ on $\TT$.
Set $b:=a\phi$. Then $b$ lies in the unit ball of $H^\infty$ and $a$ is its Pythagorean mate,
so $b$ is non-extreme. 
By Theorem~\ref{T:non-extreme}, the space $\cH(b)$ contains the polynomials as a dense subspace.
Further, we have $b/a=\phi\in H^2$, so by Theorem~\ref{T:tame} we have $\sup_{n\ge0}\|z^n\|_{\cH(b)}<\infty$.
Also, since $\gamma\in\ell^2(\ZZ^+,\RR)$,
the function $\phi$ has real Taylor coefficients. 
This implies that $|a(z)|=|a(\overline{z})|$ a.e.\ on $\TT$,
from which it follows that $a$ has real coefficients and hence $b$ too.

Now define
\[
g(z):=\sum_{k\ge0}\alpha_kz^k.
\]
Since $\alpha\in\ell^2(\ZZ^+)\setminus\{0\}$, we have $g\in H^2\setminus\{0\}$.
Also, using \eqref{E:cond1}, we have
\[
\sum_{k\ge0}\Bigl|\sum_{j\ge0}\overline{\hat{\phi}(j)}\hat{g}(j+k)\Bigr|^2
=\sum_{k\ge0}\Bigl|\sum_{j\ge0}\gamma_j\alpha_{j+k}\Bigr|^2
=\sum_{k\ge0}|\beta_k|^2<\infty,
\]
so by Theorem~\ref{T:formula} we deduce that $g\in\cH(b)\setminus\{0\}$. 
From Corollary~\ref{C:formula}, we again have
\[
\langle g,z^j\rangle_{\cH(b)}=\alpha_j+\sum_{0\le k\le j}\beta_k\gamma_{j-k} \quad(j\in\ZZ^+),
\] 
whence, by \eqref{E:cond2}, 
\[
\langle g,z^j\rangle_{\cH(b)}=0 \quad(j\in\supp\hat{g}).
\]
By Theorem~\ref{T:Hilbert}, it follows that there exists $f\in\cH(b)$ lying outside 
the closed linear span of $\{s_n(f):n\ge0\}$ 
(and in fact the proof of Theorem~\ref{T:Hilbert} shows that we may take $f=g$).
This completes the proof.
\end{proof}

We are thus led to consider the following problem.

\begin{problem}\label{Pb:Anal1}
Show that there exist sequences $\alpha,\beta\in\ell^2(\ZZ^+,\CC)$ and $\gamma\in\ell^2(\ZZ^+,\RR)$, none of them identically zero, such that
\eqref{E:cond1} and \eqref{E:cond2} hold.
\end{problem}

Despite its elementary appearance, this problem seems 
to be quite difficult to solve.  
It is perhaps worthwhile to reflect a moment on why this should be.
The following proposition imposes a number of constraints
on possible solutions $\alpha,\beta,\gamma$,
that explain why finding them might 
be a delicate matter.

\begin{proposition}\label{P:delicate}
Suppose that $\alpha,\beta\in\ell^2(\ZZ^+,\CC)\setminus\{0\}$
and $\gamma\in\ell^2(\ZZ^+,\RR)\setminus\{0\}$
satisfy \eqref{E:cond1} and \eqref{E:cond2}.
Then:
\begin{enumerate}[\normalfont(i)]
\item $\displaystyle\sum_{k\ge0}\sum_{j\ge k}\overline{\alpha}_j\beta_k\gamma_{j-k}>0$
but $\displaystyle\sum_{j\ge0}\sum_{k\le j}\overline{\alpha}_j\beta_k\gamma_{j-k}<0$.
\item $\displaystyle\sum_{j\ge0}\sum_{0\le k\le j}|\alpha_j||\beta_k||\gamma_{j-k}|=\infty$.
\item $\displaystyle\sum_{k\ge0}\Bigl(\sum_{j\ge k}|\alpha_{j}||\gamma_{j-k}|\Bigr)^2=\infty$ 
and $\displaystyle\sum_{j\ge0}\Bigl(\sum_{0\le k\le j}|\beta_{k}||\gamma_{j-k}|\Bigr)^2=\infty$.
\item $\displaystyle\sum_{n\ge0}|\alpha_n|=\sum_{n\ge0}|\beta_n|=\sum_{n\ge0}|\gamma_n|=\infty$.
\item $\displaystyle\sum_{n\ge0}(n+1)|\alpha_n|^2=\infty$.
\item Both $\supp\alpha$ and $\displaystyle (\ZZ^+\setminus \supp\alpha)$ are infinite sets.
\end{enumerate}
\end{proposition}

\begin{proof}
(i)  Multiplying  \eqref{E:cond1} by $\overline{\beta}_k$, conjugating, and summing over $k$, we obtain
\[
\sum_{k\ge0}\sum_{j\ge k}\overline{\alpha}_j\beta_k\gamma_{j-k}=\sum_{k\ge0}|\beta_k|^2>0.
\]
On the other hand, multiplying \eqref{E:cond2} by $\overline{\alpha}_j$ and summing over $j$, we get
\[
\sum_{j\ge0}\sum_{k\le j}\overline{\alpha}_j\beta_k\gamma_{j-k}=-\sum_{j\ge0}|\alpha_j|^2<0.
\]

(ii) The two double series in (i) sum to different values.
This can only happen if
\begin{equation}\label{E:eq1}
\sum_{j\ge0}\sum_{0\le k\le j}|\alpha_j||\beta_k||\gamma_{j-k}|=\infty.
\end{equation}

(iii)
Applying Cauchy--Schwarz to the outer sum in \eqref{E:eq1}, we get
\[
\Bigl(\sum_{j\ge0}\sum_{0\le k\le j}|\alpha_j||\beta_k||\gamma_{j-k}|\Bigr)^2
\le \Bigl(\sum_{j\ge0}|\alpha_j|^2\Bigr)\Bigl(\sum_{j\ge0}\Bigl(\sum_{0\le k\le j}|\beta_k||\gamma_{j-k}|\Bigr)^2\Bigr).
\]
As the left-hand side is infinite and $(\alpha_j)\in\ell^2(\ZZ^+)$, it follows that
\[
\sum_{j\ge0}\Bigl(\sum_{0\le k\le j}|\beta_k||\gamma_{j-k}|\Bigr)^2=\infty.
\]
Likewise, by Cauchy--Schwarz again,
\[
\Bigl(\sum_{k\ge0}\sum_{j\ge k}|\alpha_j||\beta_k||\gamma_{j-k}|\Bigr)^2
\le \Bigl(\sum_{k\ge0}|\beta_k|^2\Bigr)\Bigl(\sum_{k\ge0}\Bigl(\sum_{j\ge k}|\alpha_j||\gamma_{j-k}|\Bigr)^2\Bigr).
\]
As the left-hand side is infinite and $(\beta_k)\in\ell^2(\ZZ^+)$, it follows that
\begin{equation}\label{E:eq2}
\sum_{k\ge0}\Bigl(\sum_{j\ge k}|\alpha_j||\gamma_{j-k}|\Bigr)^2=\infty.
\end{equation}

(iv) Applying Cauchy--Schwarz directly to the double sum in \eqref{E:eq1}, we get
\begin{align*}
\Bigl(\sum_{j\ge0}\sum_{0\le k\le j}|\alpha_j||\beta_k||\gamma_{j-k}|\Bigr)^2
&\le\Bigl(\sum_{j\ge0}\sum_{0\le k\le j}|\alpha_j|^2|\gamma_{j-k}|\Bigr)\Bigl(\sum_{j\ge0}\sum_{0\le k\le j}|\beta_k|^2|\gamma_{j-k}|\Bigr)\\
&\le \Bigl(\sum_{n\ge0}|\gamma_n|\Bigr)^2\Bigl(\sum_{j\ge0}|\alpha_j|^2\Bigr)\Bigl(\sum_{k\ge0}|\beta_k|^2\Bigr).
\end{align*}
As the left-hand side is infinite and both $(\alpha_j),(\beta_k)\in\ell^2(\ZZ^+)$, it follows that
\[
\sum_{n\ge0}|\gamma_n|=\infty.
\]
The proofs that $\sum_{n\ge0}|\alpha_n|=\infty$ and $\sum_{n\ge0}|\beta_n|=\infty$ are similar.

(v)  The relation \eqref{E:eq2} can be rewritten as
\[
\sum_{k\ge0}\Bigl(\sum_{j\ge0}|\alpha_{j+k}||\gamma_j|\Bigr)^2=\infty.
\]
Applying Cauchy--Schwarz to the inner sum, we get
\[
\sum_{k\ge0}\Bigl(\sum_{j\ge0}|\alpha_{j+k}||\gamma_j|\Bigr)^2
\le
\sum_{k\ge0}\Bigl(\sum_{j\ge0}|\alpha_{j+k}|^2\Bigr)\Bigl(\sum_{j\ge0}|\gamma_j|^2\Bigr).
\]
As the left-hand side is infinite and  $(\gamma_j)\in\ell^2(\ZZ^+)$,
it follows that
\[
\sum_{k\ge0}\sum_{j\ge0}|\alpha_{j+k}|^2=\infty,
\]
in other words, that
\[
\sum_{n\ge0}(n+1)|\alpha_{n}|^2=\infty.
\]

(vi) Since $\sum_{n\ge0}|\alpha_n|=\infty$, it is obvious that $\supp\alpha$ is infinite.
To prove that $\ZZ^+\setminus\supp\alpha$ is infinite,
we note that the proof of Theorem~\ref{T:equiv} shows that
$\supp\alpha=\supp\hat{g}$, where 
 $g$ is a function in a de Branges--Rovnyak space $\cH(b)$ 
in which polynomials are dense, and such that
 $\langle g,z^j\rangle_{\cH(b)}=0$ for all $j\in \supp\hat{g}$.
The result therefore follows from the last part of  Theorem~\ref{T:Hilbert}.   
\end{proof}


\section{Strategy for the solution of Problem~\ref{Pb:Anal1}}\label{S:strategy}

The rest of the paper is taken up with the solution
of Problem~\ref{Pb:Anal1}.
In this section, we formulate the overall strategy
for finding the solution.

Let $\gamma\in\ell^2(\ZZ^+,\RR)$ and let $R\subset\ZZ^+$.
Associated to $(\gamma,R)$, we define the unbounded
operator
$\Gamma:c_{00}(\ZZ^+)\subset\ell^2(\ZZ^+)\to\ell^2(R)$ by
\[
(\Gamma x)_r:=\sum_{k\le r}\gamma_{r-k}x_k
\quad(x\in c_{00}(\ZZ^+),\,r\in R),
\]
and its formal transpose 
$\Gamma^t:c_{00}(R)\subset\ell^2(R)\to\ell^2(\ZZ^+)$ by
\[
(\Gamma^t y)_k:=\sum_{\substack{r\in R\\r\ge k}}\gamma_{r-k}y_r
\quad(y\in c_{00}(R),\, k\in \ZZ^+).
\]
Also, on $\cH:=\ell^2(\ZZ^+)\oplus \ell^2(R)$, we define the unbounded operator $D_\Gamma:c_{00}(\ZZ^+)\oplus c_{00}(R)\subset\cH\to \cH$ by
\[
D_\Gamma(x,y):=(\Gamma^t y,\Gamma x)\quad(x\in c_{00}(\ZZ^+),\, y\in c_{00}(R)).
\]
As $D_\Gamma$ is densely defined, it has an adjoint $D_\Gamma^*$.

The following lemma transforms the question of solving
Problem~\ref{Pb:Anal1} into one of showing that
$\ker(D_\Gamma^*-i I)\ne\{0\}$.

\begin{lemma}\label{L:ker}
If
$\ker(D_\Gamma^*-i I)\ne\{0\}$,
then there exist non-zero $\alpha,\beta\in\ell^2(\ZZ^+,\CC)$
with $\supp\alpha\subset R$ such that \eqref{E:cond1} and \eqref{E:cond2} both hold.
\end{lemma}

\begin{proof}
Let $(u,v)\in\ell^2(\ZZ^+)\oplus\ell^2(R)$ 
be a non-zero element of $\ker(D_\Gamma^*-iI)$.
Then, for all $(x,y)\in c_{00}(\ZZ^+)\oplus c_{00}(R)$, we have
\begin{align*}
\Bigl\langle i(u,v),\,(x,y)\Bigr\rangle_{\cH}
&=\Bigl\langle D_\Gamma^*(u,v),\,(x,y)\Bigr\rangle_{\cH}\\
&=\Bigl\langle (u,v),\,D_\Gamma(x,y)\Bigr\rangle_{\cH}\\
&=\Bigl\langle (u,v),\,(\Gamma^t y,\Gamma x)\Bigr\rangle_{\cH}.
\end{align*}
In particular, taking $x=e_k~(k\in\ZZ^+)$ and $y=0$, we obtain
\begin{equation}\label{E:iuk}
iu_k=\sum_{\substack{r\in R\\r\ge k}}\gamma_{r-k}v_r \quad(k\in\ZZ^+),
\end{equation}
and, taking $x=0$ and $y=e_r~(r\in R)$, we obtain
\begin{equation}\label{E:ivr}
iv_r=\sum_{\substack{k\in\ZZ^+\\k\le r}}\gamma_{r-k}u_k \quad(r\in R).
\end{equation}
In particular, both $u$ and $v$ are non-zero, since either
one being zero forces the other to be zero.

Extend $v$ to $\ZZ^+$ by setting  it equal to zero outside $R$,
and call the resulting sequence $\alpha$. Also, set $\beta:=iu$.
Then both $\alpha,\beta\in\ell^2(\ZZ^+,\CC)$, neither is identically zero, 
$\supp\alpha\subset R$, and the relations \eqref{E:iuk}
and \eqref{E:ivr} imply \eqref{E:cond1} and \eqref{E:cond2} respectively.
\end{proof}

So our problem now becomes one of constructing 
$\gamma$ and $R$ so that the resulting operator $D_\Gamma$
satisfies $\ker(D_\Gamma^*-iI)\ne\{0\}$. 
This seems to be  difficult to do directly, so instead we approach 
it indirectly, using a perturbation argument.

Suppose that we can write $\Gamma=\Delta+\cE$,
where $\Delta,\cE:c_{00}(\ZZ^+)\subset\ell^2(\ZZ^+)\to \ell^2(R)$ are operators with real matrices such that
$\cE$ is bounded. 
Let $\Delta^t$ be the formal transpose of $\Delta$, and let $D_\Delta$ be the densely defined operator given by
\[
D_\Delta(x,y):=(\Delta^t y,\Delta x)\quad(x\in c_{00}(\ZZ^+),\, y\in c_{00}(R)).
\]

\begin{lemma}\label{L:perturb}
If $\ker(D_\Delta^*-iI)\ne\{0\}$,
then $\ker(D_\Gamma^*-iI)\ne\{0\}$.
\end{lemma}

To prove this lemma, we need a little operator theory.
Recall that a densely defined operator  $T$ on a Hilbert space $H$
is called \emph{symmetric} if $T\subset T^*$ and \emph{self-adjoint} if $T=T^*$. A densely defined symmetric operator $T$ is \emph{essentially self-adjoint} if its closure $\overline{T}$ is self-adjoint 
(equivalently, if $\overline{T}=T^*$).
We cite two basic results:
\begin{itemize} 
\item A densely defined, symmetric operator $T$ is essentially self-adjoint iff $\ker(T^*-iI)=\{0\}$ and $\ker(T^*+iI)=\{0\}$ (see \cite[Corollary~2.8]{Te14}).
\item If $T$ is essentially self-adjoint and $S$ is a bounded self-adjoint operator on $H$, then $T+S$ is essentially self-adjoint
(see \cite[Theorem~6.4]{Te14}).
\end{itemize}

\begin{proof}[Proof of Lemma~\ref{L:perturb}]
A  routine verification shows that $D_\Gamma$ and $D_\Delta$ are symmetric operators.
On $c_{00}(\ZZ^+)\oplus c_{00}(R)$, we have
\[
D_\Gamma-D_\Delta=
\begin{pmatrix}
0 &\cE^t\\\cE &0
\end{pmatrix},
\]
which is the restriction to 
$c_{00}(\ZZ^+)\oplus c_{00}(R)$
of a bounded self-adjoint operator,
since $\cE$ is real and bounded.
Thus $D_\Gamma$ is essentially self-adjoint if and only if $D_\Delta$ is.
By assumption, $\ker(D_\Delta^*-iI)\ne\{0\}$, so $D_\Delta$ is not essentially self-adjoint. 
Therefore $D_\Gamma$ is not essentially self-adjoint either. 
It follows that at least one of the spaces
$\ker(D_\Gamma^*-iI)$ and $\ker(D_\Gamma^*+iI)$ is non-zero.
As $\gamma$ is a \emph{real} sequence, the conjugation map $(x,y)\mapsto(\overline{x},\overline{y})$ is an antilinear bijection  between $\ker(D_\Gamma^*-iI)$ and $\ker(D_\Gamma^*+iI)$, so, if one of these spaces is
non-zero, then so is the other. We conclude that $\ker(D_\Gamma^*-iI)\ne\{0\}$, as desired.
\end{proof}

We can now enunciate our general strategy.
Having chosen $\gamma$ and $R$ appropriately, 
we shall first form $\Gamma$ as above, and then derive $\Delta$
from it by deleting 
certain entries in its matrix. Our aim is to do this
in such a way that:
\begin{itemize}
\item 
we can see directly that $\ker(D_\Delta^*-iI)\ne\{0\}$
by exhibiting an explicit eigenvector;
\item the error $\cE:=\Gamma-\Delta$ introduced in performing the deletion is a \emph{bounded} operator from $\ell^2(\ZZ^+)$ to $\ell^2(R)$.
\end{itemize}
If we succeed in these aims, then Lemmas~\ref{L:ker}
and \ref{L:perturb} will finish the job.

For all of this to work, great care needs to be taken in the construction of
$\gamma$ and $R$. This is our next task.


\section{Construction of $\gamma,R,\Gamma,\Delta,\cE$}\label{S:construction}

The construction is combinatorial in nature, and will be carried out in
several steps.


\subsection{The sequences $(\mu_n)$ and $(\nu_n)$}

For $n\ge0$, we set
\[
\mu_n:=\frac{4}{3}4^n-\frac{4}{3}
\quad\text{and}\quad
\nu_n:=\frac{5}{3}-\frac{8}{3}4^n.
\]
The reason for this particular choice will become apparent later.
For the moment, we simply observe that $\mu_n$ and $\nu_n$ are
integers satisfying the relations
\begin{equation}\label{E:recurr}
2\mu_n+\nu_n=-1
\quad\text{and}\quad
\nu_n+\frac{1}{2}\mu_{n+1}=1 
\quad(n\ge0).
\end{equation}


\subsection{The numbers $N_n$}\label{S:AB}

Fix integers $A,B$ satisfying $A>16$ and $B>16A^2$.
For $n\ge0$, we define
\[
N_n:=
\begin{cases}
A^n, &n \text{~odd},\\
B^n, &n \text{~even}.
\end{cases}
\]
With this choice of the numbers $(N_n)$, we have
\begin{equation}\label{E:seriesconv}
\sum_{n\ge0}\frac{|\mu_n|^2+|\nu_n|^2}{N_n}<\infty,
\quad
\sum_{j\ge0}\frac{N_{2j+1}^2|\nu_{2j}|^2}{N_{2j}}<\infty,
\qquad
\sum_{j\ge0}\frac{N_{2j+1}^2|\mu_{2j+2}|^2}{N_{2j+2}}<\infty.
\end{equation}


\subsection{The matrices $Q_n$ and the vectors $u_n,v_n$}
For $n\ge0$, set
\[
X_n:=\{x\in\ZZ: 0\le x<N_n\}
\quad\text{and}\quad
Y_n:=\{y\in\ZZ: -N_n<y<N_n\}.
\]
Define $Q_n:X_n\times Y_n\to\{0,1\}$ by
\[
Q_n(x,y):=1_{X_n}(x-y)
\quad(x\in X_n,\,y\in Y_n).
\]
Thus $Q_n$ is a $N_n\times (2N_n-1)$ Toeplitz matrix with rows labeled by $X_n$ and columns by $Y_n$:
\[
Q_n=
\begin{pmatrix}
1 &1  &\dots &1 &1 &0 &\dots  &0 &0\\
0 &1  &\dots &1 &1 &1 &\dots &0 &0 \\
\vdots&\vdots  &\ddots &\vdots &\vdots &\vdots &\ddots &\vdots &\vdots\\
0 &0 &\dots &1 &1 &1 &\dots &1 &0\\
0 &0 &\dots &0 &1 &1 &\dots &1 &1
\end{pmatrix}.
\]
As such, it can be identified
with an operator $Q_n:\ell^2(Y_n)\to\ell^2(X_n)$.
We define $\sigma_n$ to be its operator norm, namely
\[
\sigma_n:=\|Q_n:\ell^2(Y_n)\to\ell^2(X_n)\|.
\]

The following lemma provides a useful lower bound for $\sigma_n$.
 
\begin{lemma}\label{L:sigmabound}
$\sigma_n\ge \sqrt{2/3}N_n$.
\end{lemma}

\begin{proof}
Let $u:=(1,1,\dots,1)\in\ell^2(X_n)$. Then $\|u\|_{\ell^2(X_n)}^2=N_n$.
Also, we have
$(Q_n^*u)_y=N_n-|y|$ for all $y\in Y_n$,
so
\[
\|Q_n^*u\|_{\ell^2(Y_n)}^2
=\sum_{y\in Y_n}(N_n-|y|)^2
=N_n^2+2\sum_{j=1}^{N_n-1}j^2
=\frac{2N_n^3+N_n}{3}
\ge\frac{2}{3}N_n^3.
\]
Hence 
\[
\|Q^*_n:\ell^2(X_n)\to\ell^2(Y_n)\|^2
\ge\frac{\|Q_n^*u\|_{\ell^2(Y_n)}^2}{\|u\|_{\ell^2(X_n)}^2}
\ge \frac{(2/3)N_n^3}{N_n}=(2/3)N_n^2. 
\]
Since $Q_n$ and $Q_n^*$ have the same operator
norm, the result follows.
\end{proof}

As $Q_n$ is a real matrix with $\|Q_n\|=\sigma_n$, 
there exist real unit vectors $u_n\in\ell^2(X_n)$
and $v_n\in\ell^2(Y_n)$ such that
\[
Q_nv_n=\sigma_nu_n
\quad\text{and}\quad
Q_n^*u_n=\sigma_n v_n.
\]
We fix a choice of such vectors for each $n\ge0$.


\subsection{The kernels $G_n^{(0)}$ and $G_n^{(1)}$}

For $n\ge0$, we define $G_n^{(0)}:\ZZ^4\to\RR$ by
\[
G_n^{(0)}(z_n^-,z_n^+,z_{n+1}^-,z_{n+1}^+):=
\frac{\nu_n}{\sigma_n}1_{\{0\}}(z_n^-)1_{X_n}(z_n^+)v_{n+1}(-z_{n+1}^-)u_{n+1}(-z_{n+1}^+),
\]
where the vectors $u_{n+1},v_{n+1}$ are understood to be zero off their indexing sets.
Likewise, we define $G_n^{(1)}:\ZZ^4\to\RR$ by
\[
G_n^{(1)}(z_n^-,z_n^+,z_{n+1}^-,z_{n+1}^+):=
\frac{\mu_{n+1}}{\sigma_{n+1}}u_n(-z_n^-)v_n(-z_n^+)1_{X_{n+1}}(z_{n+1}^-)1_{\{0\}}(z_{n+1}^+).
\]

A simple calculation gives
\begin{equation}\label{E:G0norm}
\|G_n^{(0)}\|_{\ell^2(\ZZ^4)}^2
=\frac{|\nu_n|^2}{\sigma_n^2}1\cdot|X_n|\cdot\|v_{n+1}\|_{\ell^2(Y_{n+1})}^2\cdot\|u_{n+1}\|_{\ell^2(X_{n+1})}^2
=\frac{|\nu_n|^2N_n}{\sigma_n^2},
\end{equation}
and likewise
\begin{equation}\label{E:G1norm}
\|G_n^{(1)}\|_{\ell^2(\ZZ^4)}^2
=\frac{|\mu_{n+1}|^2}{\sigma_{n+1}^2}\|u_n\|_{\ell^2(X_n)}^2\cdot\|v_n\|_{\ell^2(Y_n)}^2\cdot|X_{n+1}|\cdot1
=\frac{|\mu_{n+1}|^2N_{n+1}}{\sigma_{n+1}^2}.
\end{equation}


\subsection{Integer encoding and the decoding principle}\label{S:encoding}
The combinatorial models developed thus far
need to be encoded into the positive integers.
We now set up this encoding in such a way that no information is lost.

Let $B$ be the positive integer selected in \S\ref{S:AB}.
For $m\ge0$, set 
\[
c_m:=(8B)^{2^{m+1}}.
\]
Relabel the numbers $c_0,c_1,c_2,\dots$ using the following symbols, in order:
\[
H_0^-,H_0^+,H_1^-,H_1^+,Q_0,P_0,
\]
and then 
\[
H_{2j}^-,H_{2j}^+,H_{2j+1}^-,H_{2j+1}^+,Q_{2j-1},Q_{2j},P_{2j},P_{2j-1},
\]
successively for $j=1,2,3,\dots$.
Thus every symbol $H_n^{\pm},Q_n,P_n$ is one of the $c_m$,
and the ordering of these symbols specified above corresponds to the natural ordering of the $c_m$.

In what follows,  we shall consider finite sums of the form 
\[
\sum_n (x_n^{-}H_n^{-} +x_n^{+}H_n^{+}+p_nP_n+q_nQ_n),
\]
where the coefficients $x_n^{-},x_n^{+},p_n,q_n$ are integers (positive or negative).
If such a sum has some non-zero coefficients,  then the one associated to the largest symbol 
is called the \emph{principal coefficient}.

\begin{lemma}[Decoding principle]\label{L:decode}
Let $\sum_n (x_n^{-}H_n^{-} +x_n^{+}H_n^{+}+p_nP_n+q_nQ_n)$
be a finite sum with integer coefficients satisfying $|x_n^{\pm}|\le 4N_n$ and $|p_n|,|q_n|\le 4$
for all $n\ge0$.
\begin{enumerate}[\normalfont(i)]
\item
 If the sum has some non-zero coefficients, then it has the same sign as its principal coefficient. 
 \item If the sum is equal to zero, then all its coefficients are zero.
 \end{enumerate}
\end{lemma}

\begin{proof}
Recall that each symbol $H_n^{\pm},Q_n,P_n$ is one of the numbers $c_m$. If $c_m$ occurs in the sum, then either it corresponds to
some $H_n^{\pm}$ with $n\le m$,   so its coefficient has absolute
value at most $4N_n\le 4B^n\le 4B^m$, or it corresponds to
some $P_n$ or $Q_n$, so its coefficient  has absolute value at most $4$. Either way, its coefficient has absolute value at most $4B^{m+1}$. Thus the sum in question has the form
$\sum_{m=0}^M\lambda_mc_m$, where $\lambda_m$ are integers satisfying
$|\lambda_m|\le 4B^{m+1}$ for all $m$.

If the $\lambda_m$ are not all zero, 
then we may as well choose $M$ so that $\lambda_M$ is the principal
coefficient.
Since $\lambda_M$ is an integer, we have $|\lambda_M|\ge1$.
Then
\begin{align*}
\Bigl|\sum_{m=0}^{M-1}\lambda_mc_m\Bigr|
&\le \sum_{m=0}^{M-1} 4B^{m+1}(8B)^{2^{m+1}}
\le 4MB^M(8B)^{2^M}\\
&<(8B)^{2^{M+1}}=c_M\le|\lambda_Mc_M|.
\end{align*}
Thus $\sum_{m=0}^M\lambda_mc_m$ has the same sign as $\lambda_Mc_M$. As $c_M>0$, this is the same sign as the principal coefficient $\lambda_M$.
This proves part (i) of the lemma. Part~(ii)
is an immediate consequence.
\end{proof}


\subsection{The support sets $D_n^{(0)},D_n^{(1)}$ and the sequence $\gamma$}

For $n\ge0$, define
\begin{align*}
D_n^{(0)}&:=\Bigl\{P_n-Q_n+z_n^-H_n^-+z_n^+H_n^++z_{n+1}^-H_{n+1}^-+z_{n+1}^+H_{n+1}^+:\\
&\qquad\qquad G_n^{(0)}(z_n^-,z_n^+,z_{n+1}^-,z_{n+1}^+)\ne0\Bigr\},\\
D_n^{(1)}&:=\Bigl\{P_{n+1}-Q_n+z_n^-H_n^-+z_n^+H_n^++z_{n+1}^-H_{n+1}^-+z_{n+1}^+H_{n+1}^+:\\
&\qquad\qquad G_n^{(1)}(z_n^-,z_n^+,z_{n+1}^-,z_{n+1}^+)\ne0\Bigr\}.
\end{align*}

Note that, if $(z_n^-,z_n^+,z_{n+1}^-,z_{n+1}^+)$
lies in the support of either $G_n^{(0)}$ or $G_n^{(1)}$,
then necessarily 
$|z_n^{\pm}|\le N_n$ and $|z_{n+1}^{\pm}|\le N_{n+1}$.
The decoding principle Lemma~\ref{L:decode} 
therefore shows that the 
elements of $D_n^{(0)}$ and $D_n^{(1)}$ are all positive, that their 
coordinate representations are unique, and that the sets $D_n^{(0)}, D_n^{(1)}~(n\ge0)$ are all pairwise disjoint.

For $d\in D_n^{(0)}$, say $d=P_n-Q_n+z_n^-H_n^-+z_n^+H_n^++z_{n+1}^-H_{n+1}^-+z_{n+1}^+H_{n+1}^+$, we define
\[
\gamma_d:=G_n^{(0)}(z_n^-,z_n^+,z_{n+1}^-,z_{n+1}^+).
\]
Likewise, for $d\in D_n^{(1)}$, say
$d=P_{n+1}-Q_n+z_n^-H_n^-+z_n^+H_n^++z_{n+1}^-H_{n+1}^-+z_{n+1}^+H_{n+1}^+$,
we define
\[
\gamma_d:=G_n^{(1)}(z_n^-,z_n^+,z_{n+1}^-,z_{n+1}^+).
\]
Finally, for $d\in\ZZ^+\setminus\cup_{n\ge0}(D_n^{(0)}\cup D_n^{(1)})$, we set 
\[
\gamma_d:=0.
\]
For convenience, we also adopt the convention that
$\gamma_d=0$ for $d<0$.

Clearly $\gamma$ is real-valued.
We also note that
\begin{align*}
\sum_{d\in\ZZ^+}|\gamma_d|^2
&=\sum_{n\ge0}\sum_{d\in D_n^{(0)}}|\gamma_d|^2+
\sum_{n\ge0}\sum_{d\in D_n^{(1)}}|\gamma_d|^2\\
&=\sum_{n\ge0}\|G_n^{(0)}\|_{\ell^2(\ZZ^4)}^2+\sum_{n\ge0}\|G_n^{(1)}\|_{\ell^2(\ZZ^4)}^2\\
&=\sum_{n\ge0}\frac{|\nu_n|^2N_n}{\sigma_n^2}
+\sum_{n\ge0}\frac{|\mu_{n+1}|^2N_{n+1}}{\sigma_{n+1}^2}\\
&\le  \sum_{n\ge0}\frac{3|\nu_n|^2}{2N_n}+\sum_{n\ge0}\frac{3|\mu_{n+1}|^2}{2N_{n+1}}<\infty,
\end{align*}
where we have used successively the definition of $\gamma$,
the relations \eqref{E:G0norm} and \eqref{E:G1norm},
Lemma~\ref{L:sigmabound}, and the convergence of 
the first series in \eqref{E:seriesconv}.
Thus $\gamma\in\ell^2(\ZZ^+,\RR)$.
Finally it is clear that $\gamma\not\equiv0$.


\subsection{The row and column index sets $R_n,K_n$
and the set $R$}

For $n\ge0$, we define the row index set $R_n$ by
\[
R_n:=\Bigl\{P_n+x_n^-H_n^-+x_n^+H_n^+: x_n^-,x_n^+\in X_n\Bigr\}.
\]
Clearly $R_n$ is a finite set of integers. The decoding principle Lemma~\ref{L:decode} shows that all the elements of $R_n$ are positive,
that their coordinate representations are unique,
and that the sets $R_n$ are pairwise disjoint.

Likewise, for each $n\ge0$, we define the column index set $K_n$ by
\begin{align*}
K_n&:=\Bigl\{Q_n+y_n^-H_n^-+y_n^+H_n^++y_{n+1}^-H_{n+1}^-+y_{n+1}^+H_{n+1}^+: \\
&\qquad\qquad y_n^-\in X_n,\, y_n^+\in Y_n,\, y_{n+1}^-\in Y_{n+1}, \, y_{n+1}^+\in X_{n+1}\Bigr\}.
\end{align*}
Again by the decoding principle, all the elements of $K_n$ are positive, 
their coordinate representations are unique,  and the sets $K_n$ 
are pairwise disjoint.

Finally in this section, we define
\[
R:=\bigcup_{n\ge0}R_n.
\]


\subsection{The operators $\Gamma,\Delta,\cE$}\label{S:GDE}

Now that $\gamma$ and $R$ are defined, we can construct
the operators $\Gamma,\Delta,\cE$, as proposed in
\S\ref{S:strategy}. 

First of all, we define $\Gamma:c_{00}(\ZZ^+)\subset\ell^2(\ZZ^+)\to\ell^2(R)$ to be the (unbounded) operator with matrix
\[
\Gamma_{r,k}:=\gamma_{r-k} \quad(r\in R,\,k\in\ZZ^+).
\]

Next, to construct $\Delta$, we first introduce some auxiliary sets as follows. For each $n\ge0$, put
\begin{align*}
S_n^{(0)}&:=\{(r,k)\in  R_n\times \ZZ^+: r-k\in D_n^{(0)}\},\\
S_n^{(1)}&:=\{(r,k)\in  R_{n+1}\times \ZZ^+ : r-k\in D_n^{(1)}\}.
\end{align*}
Define $\Delta:c_{00}(\ZZ^+)\subset\ell^2(\ZZ^+)\to\ell^2(R)$
to be the (unbounded) operator whose matrix is given by
\[
\Delta_{r,k}:=
\begin{cases}
\Gamma_{r,k}, &(r,k)\in \cup_{n\ge0}(S_n^{(0)}\cup S_n^{(1)})\\
0, &(r,k)\in (R\times\ZZ^+)\setminus \cup_{n\ge0}(S_n^{(0)}\cup S_n^{(1)}).
\end{cases}
\]

Lastly, we define $\cE:c_{00}(\ZZ^+)\subset\ell^2(\ZZ^+)\to\ell^2(R)$ simply by setting
\[
\cE:=\Gamma-\Delta.
\]

All the objects $\gamma,R,\Gamma,\Delta,\cE$ have now been constructed.
Following the strategy proposed in \S\ref{S:strategy}, it remains
to show that
\begin{itemize}
\item $\ker(D_\Delta^*-iI)\ne\{0\}$, and
\item $\cE$ is a bounded operator.
\end{itemize}
This is the subject of the next two sections.


\section{The operator $\Delta$}

In this section we  study the operator $\Delta$
with a view to showing that 
$\ker(D_\Delta^*-iI)\ne\{0\}$.
We shall do this by exhibiting an explicit
eigenvector of $D_\Delta^*$ with eigenvalue $i$.
The construction of the eigenvector proceeds through several  lemmas.

We begin with a technical lemma on the sets $S_n^{(0)},S_n^{(1)}$
defined in the previous section.

\begin{lemma}\label{L:saturation}
For all $n\ge0$, we have
$S_n^{(0)}\subset R_n\times K_n$
and $S_n^{(1)}\subset R_{n+1}\times K_n$.
\end{lemma}

\begin{proof}
Let $(r,k)\in S_n^{(0)}$. Then $r\in R_n$, so
\[
r=P_n+x_n^-H_n^-+x_n^+H_n^+,
\]
where $x_n^-,x_n^+\in X_n$. Also $r-k\in D_n^{(0)}$, so
\[
r-k=P_n-Q_n+z_n^-H_n^-+z_n^+H_n^++z_{n+1}^-H_{n+1}^-+z_{n+1}^+H_{n+1}^+,
\]
where $(z_n^-,z_n^+,z_{n+1}^-,z_{n+1}^+)\in \supp G_n^{(0)}\subset\{0\}\times X_n\times (-Y_{n+1})\times(-X_{n+1})$.
It follows that
\[
k = Q_n+x_n^-H_n^-+(x_n^+-z_n^+)H_n^+-z_{n+1}^-H_{n+1}^--z_{n+1}^+H_{n+1}^+,
\]
where $x_n^-\in X_n$ and $(x_n^+-z_n^+)\in (X_n-X_n)=Y_n$ and $-z_{n+1}^-\in Y_{n+1}$ and $-z_{n+1}^+\in X_{n+1}$. From the definition of~$K_n$,
we deduce that $k\in K_n$. Hence $S_n^{(0)}\subset R_n\times K_n$.

Now let $(r,k)\in S_n^{(1)}$. Then $r\in R_{n+1}$, so
\[
r=P_{n+1}+x_{n+1}^-H_{n+1}^-+x_{n+1}^+H_{n+1}^+,
\]
where $x_{n+1}^-,x_{n+1}^+\in X_{n+1}$. Also $r-k\in D_n^{(1)}$, so
\[
r-k=P_{n+1}-Q_n+z_n^-H_n^-+z_n^+H_n^++z_{n+1}^-H_{n+1}^-+z_{n+1}^+H_{n+1}^+,
\]
where $(z_n^-,z_n^+,z_{n+1}^-,z_{n+1}^+)\in \supp G_n^{(1)}\subset (-X_n)\times(-Y_n)\times X_{n+1}\times\{0\}$.
It follows that
\[
k = Q_n-z_n^-H_n^--z_n^+H_n^++(x_{n+1}^--z_{n+1}^-)H_{n+1}^-+x_{n+1}^+H_{n+1}^+,
\]
where $(-z_n^-)\in X_n$ and $(-z_n^+)\in Y_n$ and $(x_{n+1}^--z_{n+1}^-)\in (X_{n+1}-X_{n+1})=Y_{n+1}$ and $x_{n+1}^+\in X_{n+1}$. From the definition of~$K_n$,
we deduce that $k\in K_n$. Hence $S_n^{(1)}\subset R_{n+1}\times K_n$.
\end{proof}

For each $n\ge0$, we define a vector $U_n\in c_{00}(R)$ 
as follows. If $r\in R_n$, say
\[
r=P_n+x_n^-H_n^-+x_n^+H_n^+,
\]
where $x_n^{\pm}\in X_n$, then
\[
(U_n)_r:=u_n(x_n^-)u_n(x_n^+).
\]
Otherwise we set $(U_n)_r:=0$.
The vectors $(U_n)_{n\ge0}$ are disjointly supported 
unit vectors in $\ell^2(R)$, so they form an orthonormal
sequence in $\ell^2(R)$.

Likewise, for each $n\ge0$, we define a vector $V_n\in c_{00}(\ZZ^+)$ as follows.
If $k\in K_n$, say
\[
k=Q_n+y_n^-H_n^-+y_n^+H_n^++y_{n+1}^-H_{n+1}^-+y_{n+1}^+H_{n+1}^+,
\]
where  $y_n^-\in X_n,\, y_n^+\in Y_n,\, y_{n+1}^-\in Y_{n+1}, \, y_{n+1}^+\in X_{n+1}$, then
\[
(V_n)_k:=u_n(y_n^-)v_n(y_n^+)v_{n+1}(y_{n+1}^-)u_{n+1}(y_{n+1}^+).
\]
Otherwise we set $(V_n)_k:=0$.
The vectors $(V_n)_{n\ge0}$ are disjointly supported unit vectors
in $\ell^2(\ZZ^+)$, so they form an orthonormal
sequence in $\ell^2(\ZZ^+)$.

\begin{lemma}\label{L:DeltaVcalc}
For each $n\ge0$, we have
\[
\Delta V_n=\nu_n U_n+\mu_{n+1}U_{n+1}.
\]
\end{lemma}

\begin{proof}
We shall prove this equality by verifying that the
$r$-th coordinates of both sides agree for all $r\in R$.

First, suppose that $r\in R\setminus(R_n\cup R_{n+1})$. 
Then, by Lemma~\ref{L:saturation}, we have $\Delta_{r,k}=0$ for all $k\in K_n$, and so $(\Delta V_n)_r=0$.

Now suppose that $r\in R_n$. Then
\[
(\Delta V_n)_r
=\sum_{k:(r,k)\in S_n^{(0)}}\Gamma_{r,k}(V_n)_k
=\sum_{k:(r,k)\in S_n^{(0)}}\gamma_{r-k} (V_n)_k.
\]
Now, if $(r,k)\in S_n^{(0)}$, then $r\in R_n$ and $r-k\in D_n^{(0)}$, so
\begin{align*}
r&=P_n+x_n^{-}H_n^-+x_n^+H_n^+,\\
r-k&=P_n-Q_n+z_n^-H_n^-+z_n^+H_n^++z_{n+1}^-H_{n+1}^-+z_{n+1}^+H_{n+1}^+,
\end{align*}
where $x_n^\pm\in X_n$ and $(z_n^-,z_n^+,z_{n+1}^-,z_{n+1}^+)\in\supp G_n^{(0)}$.
Therefore
\begin{align*}
\gamma_{r-k}
&=G_n^{(0)}(z_n^-,z_n^+,z_{n+1}^-,z_{n+1}^+)\\
&=\frac{\nu_n}{\sigma_n}1_{\{0\}}(z_n^-)1_{X_n}(z_n^+)v_{n+1}(-z_{n+1}^-)u_{n+1}(-z_{n+1}^+).
\end{align*}
Also $k\in K_n$ by Lemma~\ref{L:saturation}, and
\begin{align*}
k&=r-(r-k)\\
&=Q_n+(x_n^--z_n^-)H_n^-+(x_n^+-z_n^+)H_n^+-z_{n+1}^-H_{n+1}^--z_{n+1}^+H_{n+1}^+,
\end{align*}
so
\[ 
(V_n)_{k}=u_n(x_n^--z_n^-)v_n(x_n^+-z_n^+)v_{n+1}(-z_{n+1}^-)u_{n+1}(-z_{n+1}^+).
\]
Combining these observations, we deduce
\begin{align*}
&\sum_{k:(r,k)\in S_n^{(0)}}\gamma_{r-k} (V_n)_k\\
&\qquad=\sum_{z_n^\pm,z_{n+1}^\pm\in\ZZ}\frac{\nu_n}{\sigma_n}\Bigl(1_{\{0\}}(z_n^-)1_{X_n}(z_n^+)v_{n+1}(-z_{n+1}^-)u_{n+1}(-z_{n+1}^+)\Bigr)\times\\
&\qquad\qquad\qquad\Bigl(u_n(x_n^--z_n^-)v_n(x_n^+-z_n^+)v_{n+1}(-z_{n+1}^-)u_{n+1}(-z_{n+1}^+)\Bigr)\\
&\qquad=\frac{\nu_n}{\sigma_n}u_n(x_n^-)\sum_{z_n^+\in\ZZ}1_{X_n}(z_n^+)v_n(x_n^+-z_n^+)\\
&\qquad=\frac{\nu_n}{\sigma_n}u_n(x_n^-)(Q_nv_n)(x_n^+)\\
&\qquad=\nu_nu_n(x_n^-)u_n(x_n^+)\\
&\qquad=\nu_n(U_n)_r.
\end{align*}
Here
the second equality used the fact that $u_{n+1},v_{n+1}$ are unit vectors,
the third one used the definition of $Q_n$, the fourth one
used the definition of $u_n,v_n$, and the fifth one used the definition of $U_n$.
Thus $(\Delta V_n)_r=\nu_n(U_n)_r$ for all 
$r\in R_n$.

Lastly suppose that $r\in R_{n+1}$. Then
\[
(\Delta V_n)_r
=\sum_{k:(r,k)\in S_n^{(1)}}\Gamma_{r,k}(V_n)_k
=\sum_{k:(r,k)\in S_n^{(1)}}\gamma_{r-k}(V_n)_k.
\]
Now, if $(r,k)\in S_n^{(1)}$, then $r\in R_{n+1}$ and $r-k\in D_n^{(1)}$, so
\begin{align*}
r&=P_{n+1}+x_{n+1}^{-}H_{n+1}^-+x_{n+1}^+H_{n+1}^+,\\
r-k&=P_{n+1}-Q_n+z_n^-H_n^-+z_n^+H_n^++z_{n+1}^-H_{n+1}^-+z_{n+1}^+H_{n+1}^+,
\end{align*}
where $x_{n+1}^\pm\in X_{n+1}$ and $(z_n^-,z_n^+,z_{n+1}^-,z_{n+1}^+)\in\supp G_n^{(1)}$.
Therefore
\begin{align*}
\gamma_{r-k}
&=G_n^{(1)}(z_n^-,z_n^+,z_{n+1}^-,z_{n+1}^+)\\
&=\frac{\mu_{n+1}}{\sigma_{n+1}}u_n(-z_n^-)v_n(-z_n^+)1_{X_{n+1}}(z_{n+1}^-)1_{\{0\}}(z_{n+1}^+).
\end{align*}
Also $k\in K_n$ by Lemma~\ref{L:saturation}, and 
\begin{align*}
k&=r-(r-k)\\
&=Q_n-z_n^-H_n^--z_n^+H_n^++(x_{n+1}^--z_{n+1}^-)H_{n+1}^-+(x_{n+1}^+-z_{n+1}^+)H_{n+1}^+,
\end{align*}
so
\[ 
(V_n)_{k}=u_n(-z_n^-)v_n(-z_n^+)v_{n+1}(x_{n+1}^--z_{n+1}^-)u_{n+1}(x_{n+1}^+-z_{n+1}^+).
\]
Combining these observations, we deduce
\begin{align*}
&\sum_{k:(r,k)\in S_n^{(1)}}\gamma_{r-k}(V_n)_k\\
&\qquad=\sum_{z_n^\pm,z_{n+1}^\pm\in\ZZ}\frac{\mu_{n+1}}{\sigma_{n+1}}\Bigl(u_n(-z_n^-)v_n(-z_n^+)1_{X_{n+1}}(z_{n+1}^-)1_{\{0\}}(z_{n+1}^+)\Bigr)\times\\
&\qquad\qquad\qquad\Bigl(u_n(-z_n^-)v_n(-z_n^+)v_{n+1}(x_{n+1}^--z_{n+1}^-)u_{n+1}(x_{n+1}^+-z_{n+1}^+)\Bigr)\\
&\qquad=\frac{\mu_{n+1}}{\sigma_{n+1}}u_{n+1}(x_{n+1}^+)\sum_{z_{n+1}^-\in\ZZ}1_{X_{n+1}}(z_{n+1}^-)v_{n+1}(x_{n+1}^--z_{n+1}^-)\\
&\qquad=\frac{\mu_{n+1}}{\sigma_{n+1}}u_{n+1}(x_{n+1}^+)(Q_{n+1}v_{n+1})(x_{n+1}^-)\\
&\qquad=\mu_{n+1}u_{n+1}(x_{n+1}^+)u_{n+1}(x_{n+1}^-)\\
&\qquad=\mu_{n+1}(U_{n+1})_r.
\end{align*}
We deduce that $(\Delta V_n)_r=\mu_{n+1}(U_{n+1})_r$ for all 
$r\in R_{n+1}$.

Putting everything together, we conclude that
\[
\Delta V_n=\nu_n U_n+\mu_{n+1}U_{n+1}.
\]
This completes the proof of the lemma.
\end{proof}

\begin{lemma}\label{L:DeltatUcalc}
For each $n\ge1$ we have
\[
\Delta^t U_n=\mu_n V_{n-1}+\nu_nV_n,
\]
and if $n=0$ then
\[
\Delta^t U_0=\nu_0V_0.
\]
\end{lemma}

\begin{proof}
This is similar to the previous proof, though there are also
some differences.

First, suppose that $k\in \ZZ^+\setminus(K_n\cup K_{n-1})$
(or that $k\in\ZZ^+\setminus K_0$ in the case $n=0$). 
Then, by Lemma~\ref{L:saturation}, we have $\Delta_{r,k}=0$ for all $r\in R_n$, and so $(\Delta^t U_n)_k=0$.

Now suppose that $k\in K_n$. Then
\[
(\Delta^t U_n)_k=\sum_{r\in R_n}\Delta_{r,k}(U_n)_r
=\sum_{r:(r,k)\in S_n^{(0)}}\Gamma_{r,k}(U_n)_r
=\sum_{r:(r,k)\in S_n^{(0)}}\gamma_{r-k}(U_n)_r.
\]
Now, if $r\in R_n$ and $k\in K_n$, then
\begin{align*}
r&=P_n+x_n^{-}H_n^-+x_n^+H_n^+,\\
k&=Q_n+y_n^-H_n^-+y_n^+H_n^++y_{n+1}^-H_{n+1}^-+y_{n+1}^+H_{n+1}^+,
\end{align*}
where $x_n^\pm\in X_n$ and $(y_n^-,y_n^+,y_{n+1}^-,y_{n+1}^+)\in X_n\times Y_n\times Y_{n+1}\times X_{n+1}$.
Then
\[
r-k=P_n-Q_n+(x_n^--y_n^-)H_n^-+(x_n^+-y_n^+)H_n^+-y_{n+1}^-H_{n+1}^--y_{n+1}^+H_{n+1}^+,
\]
and so
\begin{align*}
\gamma_{r-k}
&=G_n^{(0)}(x_n^--y_n^-,x_n^+-y_n^+,-y_{n+1}^-,-y_{n+1}^+)\\
&=\frac{\nu_n}{\sigma_n}1_{\{0\}}(x_n^--y_n^-)1_{X_n}(x_n^+-y_n^+)v_{n+1}(y_{n+1}^-)u_{n+1}(y_{n+1}^+).
\end{align*}
Also we have
\[
(U_n)_r=u_n(x_n^-)u_n(x_n^+).
\]
Combining these observations, we deduce
\begin{align*}
&\sum_{r:(r,k)\in S_n^{(0)}}\gamma_{r-k}(U_n)_r\\
&\qquad=\sum_{x_n^\pm\in X_n}\frac{\nu_n}{\sigma_n}\Bigl(1_{\{0\}}(x_n^--y_n^-)1_{X_n}(x_n^+-y_n^+)v_{n+1}(y_{n+1}^-)u_{n+1}(y_{n+1}^+)\Bigr)\times\\
&\qquad\qquad\qquad\Bigl(u_n(x_n^-)u_n(x_n^+)\Bigr)\\
&\qquad=\frac{\nu_n}{\sigma_n}v_{n+1}(y_{n+1}^-)u_{n+1}(y_{n+1}^+)u_{n}(y_n^-)\sum_{x_n^+\in X_n}1_{X_n}(x_n^+-y_n^+)u_n(x_n^+)\\
&\qquad=\frac{\nu_n}{\sigma_n}v_{n+1}(y_{n+1}^-)u_{n+1}(y_{n+1}^+)u_{n}(y_n^-)(Q_n^*u_n)(y_n^+)\\
&\qquad=\nu_n v_{n+1}(y_{n+1}^-)u_{n+1}(y_{n+1}^+)u_{n}(y_n^-)v_n(y_n^+)\\
&\qquad=\nu_n(V_n)_k.
\end{align*}
Thus $(\Delta^t U_n)_k=\nu_n(V_n)_k$ for all $k\in K_n$.

Lastly, suppose that $n\ge1$ and that $k\in K_{n-1}$. Then
\[
(\Delta^t U_n)_k=\sum_{r\in R_n}\Delta_{r,k}(U_n)_r
=\sum_{r:(r,k)\in S_{n-1}^{(1)}}\Gamma_{r,k}(U_n)_r
=\sum_{r:(r,k)\in S_{n-1}^{(1)}}\gamma_{r-k}(U_n)_r.
\]
Now, if $r\in R_n$ and $k\in K_{n-1}$, then
\begin{align*}
r&=P_n+x_n^{-}H_n^-+x_n^+H_n^+,\\
k&=Q_{n-1}+y_{n-1}^-H_{n-1}^-+y_{n-1}^+H_{n-1}^++y_{n}^-H_{n}^-+y_{n}^+H_{n}^+,
\end{align*}
where $x_n^\pm\in X_n$ and $(y_{n-1}^-,y_{n-1}^+,y_{n}^-,y_{n}^+)\in X_{n-1}\times Y_{n-1}\times Y_{n}\times X_{n}$.
Then
\[
r-k=P_n-Q_{n-1}-y_{n-1}^-H_{n-1}^--y_{n-1}^+H_{n-1}^++(x_n^--y_{n}^-)H_{n}^-+(x_n^+-y_{n}^+)H_{n}^+,
\]
and so
\begin{align*}
\gamma_{r-k}
&=G_{n-1}^{(1)}(-y_{n-1}^-,-y_{n-1}^+,(x_n^--y_n^-),(x_n^+-y_n^+))\\
&=\frac{\mu_{n}}{\sigma_{n}} u_{n-1}(y_{n-1}^-)v_{n-1}(y_{n-1}^+)1_{X_n}(x_n^--y_n^-)1_{\{0\}}(x_n^+-y_n^+).
\end{align*}
Also we have
\[
(U_n)_r=u_n(x_n^-)u_n(x_n^+).
\]
Combining these observations, we deduce
\begin{align*}
&\sum_{r:(r,k)\in S_{n-1}^{(1)}}\gamma_{r-k}(U_n)_r\\
&\qquad=\sum_{x_n^\pm\in X_n}\frac{\mu_{n}}{\sigma_{n}} \Bigl(u_{n-1}(y_{n-1}^-)v_{n-1}(y_{n-1}^+)1_{X_n}(x_n^--y_n^-)1_{\{0\}}(x_n^+-y_n^+)\Bigr)\times\\
&\qquad\qquad\qquad\Bigl(u_n(x_n^-)u_n(x_n^+)\Bigr)\\
&\qquad=\frac{\mu_n}{\sigma_n}u_{n-1}(y_{n-1}^-)v_{n-1}(y_{n-1}^+)u_{n}(y_n^+)\sum_{x_n^-\in X_n}1_{X_n}(x_n^--y_n^-)u_n(x_n^-)\\
&\qquad=\frac{\mu_n}{\sigma_n}u_{n-1}(y_{n-1}^-)v_{n-1}(y_{n-1}^+)u_{n}(y_n^+)(Q_n^*u_n)(y_n^-)\\
&\qquad=\mu_n u_{n-1}(y_{n-1}^-)v_{n-1}(y_{n-1}^+)u_{n}(y_n^+)v_n(y_n^-)\\
&\qquad=\mu_n(V_{n-1})_k.
\end{align*}
Thus $(\Delta^t U_n)_k=\mu_n(V_{n-1})_k$ for all $k\in K_{n-1}$.

Putting everything together, we conclude that
\[
\Delta^t U_n=\mu_n V_{n-1}+\nu_{n}V_{n},
\]
with the first term on the right-hand side omitted if $n=0$.
This completes the proof.
\end{proof}

As remarked earlier, $(U_n)_{n\ge0}$ and $(V_n)_{n\ge0}$ 
are orthonormal sequences in $\ell^2(R)$ and $\ell^2(\ZZ^+)$
respectively. We may therefore form vectors
$U\in\ell^2(R)$ and $V\in\ell^2(\ZZ^+)$ by setting
\[
U:=\sum_{n\ge0}2^{-n}U_n
\quad\text{and}\quad
V:=\sum_{n\ge0}2^{-n}V_n.
\]

\begin{lemma}\label{L:eigenvector}
The vector $(-iV,U)$ lies in the domain of $D_\Delta^*$ and satisfies
\[
D_\Delta^*(-iV,U)=i(-iV,U).
\]
In particular, $\ker(D_\Delta^*-iI)\ne\{0\}$.
\end{lemma}

\begin{proof}
We need to show that, for all $(x,y)\in c_{00}(\ZZ^+)\oplus c_{00}(R)$,
\[
\Bigl\langle (-iV,U),\,D_\Delta(x,y)\Bigr\rangle_{\ell^2(\ZZ^+)\oplus\ell^2(R)}=
\Bigl\langle i(-iV,U),\,(x,y)\Bigr\rangle_{\ell^2(\ZZ^+)\oplus\ell^2(R)}.
\]
This is equivalent to proving that
\begin{align*}
\langle -V,\,\Delta^ty\rangle_{\ell^2(\ZZ^+)}
&=\langle U,\,y\rangle_{\ell^2(R)} 
\quad\forall y\in c_{00}(R),\\
\intertext{and}
\langle U,\,\Delta x\rangle_{\ell^2(R)}
&=\langle V,\,x\rangle_{\ell^2(\ZZ^+)}
\quad\forall x\in c_{00}(\ZZ^+),
\end{align*}
which we shall now do.

If $y\in c_{00}(R)$, then we have
\begin{align*}
&\langle -V,\,\Delta^t y\rangle_{\ell^2(\ZZ^+)}\\
&\quad=-\sum_{n\ge0}2^{-n}\langle V_n,\,\Delta^t y\rangle_{\ell^2(\ZZ^+)}\\
&\quad=-\sum_{n\ge0}2^{-n}\langle \Delta V_n,\, y\rangle_{\ell^2(R)}\\
&\quad=-\sum_{n\ge0}2^{-n}\langle \nu_nU_n+\mu_{n+1}U_{n+1},\, y\rangle_{\ell^2(R)} 
&\text{(by Lemma~\ref{L:DeltaVcalc})}\\
&\quad=-\sum_{n\ge0}2^{-n}\langle (\nu_n+2\mu_n)U_n,\, y\rangle_{\ell^2(R)}+2\mu_0\langle U_0,y\rangle_{\ell^2(R)}\\
&\quad=-\sum_{n\ge0}2^{-n}\langle -U_n,\,y\rangle_{\ell^2(R)}+0 
&\text{(by \eqref{E:recurr})}\\
&\quad=\langle U,\,y\rangle_{\ell^2(R)}.
\end{align*}
Similarly, if $x\in c_{00}(\ZZ^+)$, then we have
\begin{align*}
&\langle U,\,\Delta x\rangle_{\ell^2(R)}\\
&\quad=\sum_{n\ge0}2^{-n}\langle U_n,\,\Delta x\rangle_{\ell^2(R)}\\
&\quad=\sum_{n\ge0}2^{-n}\langle \Delta^t U_n,\, x\rangle_{\ell^2(\ZZ^+)}\\
&\quad=\sum_{n\ge1}2^{-n}\langle (\mu_nV_{n-1}+\nu_{n}V_{n}),\, x\rangle_{\ell^2(\ZZ^+)} +\langle \nu_0V_0,\,x\rangle_{\ell^2(\ZZ^+)}
&\text{(by Lemma~\ref{L:DeltatUcalc})}\\
&\quad=\sum_{n\ge0}2^{-n}\langle(\mu_{n+1}/2+\nu_n)V_n,\, x\rangle_{\ell^2(\ZZ^+)}\\
&\quad=\sum_{n\ge0}2^{-n}\langle V_n,\,x\rangle_{\ell^2(\ZZ^+)}
&\text{(by \eqref{E:recurr})}\\
&\quad=\langle V,\,x\rangle_{\ell^2(\ZZ^+)}.
\end{align*}
This completes the proof.
\end{proof}


\section{The operator $\cE$}

In this section we study the operator $\cE$ with the goal
of proving that it is a bounded operator.
We shall achieve this by decomposing it as the sum
of two operators $\cE'$ and $\cE''$, each of which is bounded
(for different reasons).

\subsection{The operator $\cE'$}
Recall that,
in the process of defining $\Delta$  in \S\ref{S:GDE},
we made use of the auxiliary sets 
\begin{align*}
S_n^{(0)}&:=\{(r,k)\in  R_n\times \ZZ^+: r-k\in D_n^{(0)}\},\\
S_n^{(1)}&:=\{(r,k)\in  R_{n+1}\times \ZZ^+: r-k\in D_n^{(1)}\}.
\end{align*}
To define $\cE'$, we now introduce the companion sets
\begin{align*}
T_n^{(0)}&:=\{(r,k)\in  R_{n+1}\times \ZZ^+: r-k\in D_n^{(0)}\},\\
T_n^{(1)}&:=\{(r,k)\in  R_{n}\times \ZZ^+: r-k\in D_n^{(1)}\}.
\end{align*}
We define
$\cE':c_{00}(\ZZ^+)\subset\ell^2(\ZZ^+)\to\ell^2(R)$ 
to be the operator whose matrix is given by
\[
\cE'_{r,k}:=
\begin{cases}
\Gamma_{r,k}, &(r,k)\in \cup_{n\ge0}(T_n^{(0)}\cup T_n^{(1)})\\
0, &(r,k)\in (R\times\ZZ^+)\setminus \cup_{n\ge0}(T_n^{(0)}\cup T_n^{(1)}).
\end{cases}
\]

\begin{lemma}\label{L:empty}
If $n$ is odd, then $T_n^{(0)}=\emptyset$.
If $n$ is even, then $T_n^{(1)}=\emptyset$.
\end{lemma}

\begin{proof}
Suppose first that $n$ is odd and that, if possible, $(r,k)\in T_n^{(0)}$.
Then $r\in R_{n+1}$, so
\[
r=P_{n+1}+x_{n+1}^-H_{n+1}^-+x_{n+1}^+H_{n+1}^+,
\]
where $x_{n+1}^\pm\in X_{n+1}$. Also $r-k\in D_n^{(0)}$, so
\[
r-k=P_n-Q_n+z_n^-H_n^-+z_n^+H_n^++z_{n+1}^-H_{n+1}^-+z_{n+1}^+H_{n+1}^+,
\]
where $(z_n^-,z_n^+,z_{n+1}^-,z_{n+1}^+)\in\supp G_n^{(0)}
\subset \{0\}\times X_n\times(-Y_{n+1})\times(-X_{n+1})$.
Therefore
\[
k=-P_n+P_{n+1}+Q_n+ (\text{terms in~} H_n^\pm,H_{n+1}^\pm).
\]
According to the ordering defined in \S\ref{S:encoding},
since $n$ is odd, the principal term in this last sum is $-P_n$,
so, by the decoding principle Lemma~\ref{L:decode}, we have $k<0$.
On the other hand, since $(r,k)\in T_n^{(0)}$, we have $k\in \ZZ^+$,
a contradiction.
We conclude that
 $T_n^{(0)}=\emptyset$ whenever $n$ is odd.

Now suppose  that $n$ is even and that, if possible, $(r,k)\in T_n^{(1)}$.
Then $r\in R_{n}$, so
\[
r=P_{n}+x_{n}^-H_{n}^-+x_{n}^+H_{n}^+,
\]
where $x_{n}^\pm\in X_{n}$. Also $r-k\in D_n^{(1)}$, so
\[
r-k=P_{n+1}-Q_n+z_n^-H_n^-+z_n^+H_n^++z_{n+1}^-H_{n+1}^-+z_{n+1}^+H_{n+1}^+,
\]
where $(z_n^-,z_n^+,z_{n+1}^-,z_{n+1}^+)\in\supp G_n^{(1)}
\subset (-X_n)\times (-Y_n)\times X_{n+1}\times\{0\}$.
Therefore
\[
k=-P_{n+1}+P_{n}+Q_n+ (\text{terms in~} H_n^\pm,H_{n+1}^\pm).
\]
According to the ordering defined in \S\ref{S:encoding},
since $n$ is even, the principal term in this last sum is $-P_{n+1}$,
so, by the decoding principle, we have $k<0$. On the other hand, since $(r,k)\in T_n^{(1)}$, we have $k\in \ZZ^+$,
a contradiction. We conclude that $T_n^{(1)}=\emptyset$
whenever $n$ is even.
\end{proof}

\begin{lemma}\label{L:cE'bounded}
The entries in the matrix of $\cE'$ are square-summable.
Consequently $\cE'$ is a bounded operator.
\end{lemma}

\begin{proof}
From the definition of $\cE'$, we have
\begin{align*}
&\sum_{(r,k)\in R\times\ZZ^+}|\cE'_{r,k}|^2\\
&\quad=\sum_{n\ge0}\sum_{(r,k)\in T_n^{(0)}\cup T_n^{(1)}}|\Gamma_{r,k}|^2\\
&\qquad=\sum_{\substack{n\ge0\\n\text{~even}}}\sum_{(r,k)\in T_n^{(0)}}|\gamma_{r-k}|^2
+\sum_{\substack{n\ge0\\n\text{~odd}}}\sum_{(r,k)\in T_n^{(1)}}|\gamma_{r-k}|^2
&\text{(by Lemma~\ref{L:empty})}\\
&\quad\le\sum_{\substack{n\ge0\\n\text{~even}}}\sum_{r\in R_{n+1}}\sum_{d\in D_n^{(0)}}|\gamma_{d}|^2
+\sum_{\substack{n\ge0\\n\text{~odd}}}\sum_{r\in R_{n}}\sum_{d\in D_n^{(1)}}|\gamma_{d}|^2\\
&\quad=\sum_{\substack{n\ge0\\n\text{~even}}}|R_{n+1}|\,\|G_n^{(0)}\|_{\ell^2(\ZZ^4)}^2
+\sum_{\substack{n\ge0\\n\text{~odd}}}|R_{n}|\,\|G_n^{(1)}\|_{\ell^2(\ZZ^4)}^2\\
&\quad=\sum_{\substack{n\ge0\\n\text{~even}}}N_{n+1}^2\frac{N_n|\nu_n|^2}{\sigma_n^2}
+\sum_{\substack{n\ge0\\n\text{~odd}}}N_{n}^2\frac{N_{n+1}|\mu_{n+1}|^2}{\sigma_{n+1}^2}
&\text{(by \eqref{E:G0norm} and \eqref{E:G1norm})}\\
&\quad\le\frac{3}{2}\sum_{\substack{n\ge0\\n\text{~even}}}\frac{N_{n+1}^2|\nu_n|^2}{N_n}
+\frac{3}{2}\sum_{\substack{n\ge0\\n\text{~odd}}}\frac{N_{n}^2|\mu_{n+1}|^2}{N_{n+1}}
&\text{(by Lemma~\ref{L:sigmabound})}\\
&\quad=\frac{3}{2}\sum_{j\ge0}\frac{N_{2j+1}^2|\nu_{2j}|^2}{N_{2j}}
+\frac{3}{2}\sum_{j\ge0}\frac{N_{2j+1}^2|\mu_{2j+2}|^2}{N_{2j+2}}
<\infty
&\text{(by \eqref{E:seriesconv})}\\
\end{align*}
This completes the proof.
\end{proof}

\subsection{The operator $\cE''$}
The operator $\cE''$ is what is left over, namely
\[
\cE'':=\cE-\cE'=\Gamma-\Delta-\cE'.
\]
Explicitly, it is the operator 
$\cE'':c_{00}(\ZZ^+)\subset\ell^2(\ZZ^+)\to\ell^2(R)$ whose matrix is given by
\begin{equation}\label{E:cE''}
\cE''_{r,k}:=
\begin{cases}
0, &(r,k)\in \cup_{n\ge0}(S_n^{(0)}\cup S_n^{(1)}\cup T_n^{(0)}\cup T_n^{(1)})\\
\Gamma_{r,k}, &(r,k)\in (R\times\ZZ^+)\setminus \cup_{n\ge0}(S_n^{(0)}\cup S_n^{(1)}\cup T_n^{(0)}\cup T_n^{(1)}).
\end{cases}
\end{equation}

\begin{lemma}\label{L:column}
The matrix of $\cE''$ has at most one non-zero entry in each column.
\end{lemma}

\begin{proof}
Let $k\in\ZZ^+$, and suppose that $\cE''_{r,k}\ne0$ and $\cE''_{r',k}\ne0$. Our goal is to show that $r=r'$.

Since $r,r'\in R$, there exist integers $\ell,\ell'\ge0$ such that
$r\in R_\ell$ and $r'\in R_{\ell'}$. Thus
\begin{align*}
r&=P_\ell+x_\ell^-H_\ell^-+x_\ell^+H_\ell^+,\\
r'&=P_{\ell'}+y_{\ell'}^-H_{\ell'}^-+y_{\ell'}^+H_{\ell'}^+,
\end{align*}
where $|x_\ell^\pm|\le N_\ell$ and  $|y_{\ell'}^\pm|\le N_{\ell'}$.
Also, since $\cE''_{r,k}\ne0$ and $\cE''_{r',k}\ne0$,
we must have $\Gamma_{r,k}\ne0$ and $\Gamma_{r',k}\ne0$,
so there exist integers $n,n'\ge0$ such that $r-k\in D_n^{(0)}\cup D_n^{(1)}$ and $r'-k\in D_{n'}^{(0)}\cup D_{n'}^{(1)}$.
Thus 
\begin{align*}
r-k&=P_m-Q_n+z_n^-H_n^-+z_n^+H_n^++z_{n+1}^-H_{n+1}^-+z_{n+1}^+H_{n+1}^+,\\
r'-k&=P_{m'}-Q_{n'}+w_{n'}^-H_{n'}^-+w_{n'}^+H_{n'}^++w_{n'+1}^-H_{n'+1}^-+w_{n'+1}^+H_{n'+1}^+,
\end{align*} 
where 
\[
m\in\{n,n+1\}
\quad\text{and}\quad
m'\in\{n',n'+1\},
\]
and where
\[
|z_n^\pm|\le N_n, \quad |z_{n+1}^\pm|\le N_{n+1},
\quad |w_{n'}^\pm|\le N_{n'}\quad\text{and}\quad 
 |w_{n'+1}^\pm|\le N_{n'+1}.
\]
Further,
since $\cE''_{r,k}\ne0$ and $\cE''_{r',k}\ne0$,
it follows from \eqref{E:cE''} that 
$(r,k)\notin S_n^{(0)}\cup S_n^{(1)}\cup T_n^{(0)}\cup T_n^{(1)}$
and $(r',k)\notin S_{n'}^{(0)}\cup S_{n'}^{(1)}\cup T_{n'}^{(0)}\cup T_{n'}^{(1)}$,
whence 
\[
\ell\notin\{n,n+1\} 
\quad\text{and}\quad 
\ell'\notin\{n',n'+1\}.
\]
Since 
\[
r-(r-k)=k=r'-(r'-k),
\]
we have
\begin{equation}\label{E:bigeqn}
\begin{aligned}
&P_\ell-P_m+Q_n+(x_\ell^-H_\ell^-+x_\ell^+H_\ell^+)-
(z_n^-H_n^-+z_n^+H_n^++z_{n+1}^-H_{n+1}^-+z_{n+1}^+H_{n+1}^+)\\
&=P_{\ell'}-P_{m'}+Q_{n'}+(y_{\ell'}^-H_{\ell'}^-+y_{\ell'}^+H_{\ell'}^+)-
(w_{n'}^-H_{n'}^-+w_{n'}^+H_{n'}^++w_{n'+1}^-H_{n'+1}^-+w_{n'+1}^+H_{n'+1}^+).
\end{aligned}
\end{equation}
By the decoding principle Lemma~\ref{L:decode},
the $Q$-coefficients in \eqref{E:bigeqn} must agree, so 
\[
n=n'.
\]
Thus
$m,m'\in\{n,n+1\}$ and $\ell,\ell'\notin\{n,n+1\}$,
whence $\{\ell,\ell'\}\cap\{m,m'\}=\emptyset$.
As the $P$-coefficients in  \eqref{E:bigeqn} must agree,
it follows that 
\[
\ell=\ell'
\quad\text{and}\quad
m=m'.
\]
We are left with the equation
\begin{align*}
&(x_\ell^-H_\ell^-+x_\ell^+H_\ell^+)-
(z_n^-H_n^-+z_n^+H_n^++z_{n+1}^-H_{n+1}^-+z_{n+1}^+H_{n+1}^+)\\
&=(y_{\ell}^-H_{\ell}^-+y_{\ell}^+H_\ell^+)-
(w_{n}^-H_{n}^-+w_{n}^+H_{n}^++w_{n+1}^-H_{n+1}^-+w_{n+1}^+H_{n+1}^+).
\end{align*}
Since $\ell\notin\{n,n+1\}$, the decoding principle implies that
\[
x_\ell^-=y_\ell^-
\quad\text{and}\quad
x_\ell^+=y_\ell^+.
\]
Hence, finally, $r=r'$, as desired.
\end{proof}

\begin{lemma}\label{L:cE''bounded}
$\cE''$ is a bounded operator.
\end{lemma}

\begin{proof}
For $r\in R$, set 
\[
Z_r:=\{r-k:k\in\ZZ^+,\,\cE''_{r,k}\ne0\}.
\]
Then, for $x\in c_{00}(\ZZ^+)$, we have
\[
(\cE''x)_r=\sum_{k\in\ZZ^+}\cE''_{r,k}x_k=\sum_{d\in Z_r}\cE''_{r,r-d}x_{r-d}=\sum_{d\in Z_r}\gamma_d x_{r-d}.
\]
Hence, using Cauchy--Schwarz, we have
\begin{align*}
\sum_{r\in R}|(\cE''x)_r|^2
&=\sum_{r\in R}\Bigl|\sum_{d\in Z_r}\gamma_d x_{r-d}\Bigr|^2\\
&\le\sum_{r\in R}\Bigl(\sum_{d\in Z_r}|\gamma_d|^2\Bigr)\Bigl(\sum_{d\in Z_r}|x_{r-d}|^2\Bigr)\\
&\le \|\gamma\|_{\ell^2(\ZZ^+)}^2\sum_{r\in R}\sum_{d\in Z_r}|x_{r-d}|^2\\
&\le  \|\gamma\|_{\ell^2(\ZZ^+)}^2\sum_{k\in\ZZ^+}|x_k|^2,
\end{align*}
where the final inequality follows from the following consequence
of Lem\-ma~\ref{L:column}:
an integer $k\in \ZZ^+$ can be expressed in at most one way
as $r-d$, where $r\in R$ and $d\in Z_r$.
Thus we have
\[
\|\cE''x\|_{\ell^2(R)}\le  \|\gamma\|_{\ell^2(\ZZ^+)}\|x\|_{\ell^2(\ZZ^+)}
\quad(x\in c_{00}(\ZZ^+)),
\]
which shows that $\cE''$ is a bounded operator.
\end{proof}

\subsection{The operator $\cE$}

Finally, we reach the goal set at the beginning of the section.

\begin{lemma}\label{L:cEbounded}
$\cE$ is a bounded operator.
\end{lemma}

\begin{proof}
We have $\cE=\cE'+\cE''$,
 where $\cE'$ is bounded (by Lemma~\ref{L:cE'bounded}),
 and
 $\cE''$ is also bounded (by Lemma~\ref{L:cE''bounded}). 
The result follows.
\end{proof}


\section{Solution to Problem~\ref{Pb:Anal1} and proof of Theorem~\ref{T:main}}

We now have everything that we need to prove the main result.
First, the following result solves Problem~\ref{Pb:Anal1}.

\begin{theorem}\label{T:soln}
There exist sequences $\alpha,\beta\in\ell^2(\ZZ^+,\CC)$ and $\gamma\in\ell^2(\ZZ^+,\RR)$, none of them identically zero,  such that
\begin{align}
\label{E:cond1'}\sum_{j\ge k}\gamma_{j-k}\alpha_{j}&=\beta_k \quad(k\in\ZZ^+),\\ 
\label{E:cond2'}\sum_{0\le k\le j}\gamma_{j-k}\beta_k&=-\alpha_j  \quad(j\in \supp\alpha).
\end{align}
\end{theorem}

\begin{proof}
Construct $\gamma,R,\Gamma,\Delta,\cE$ as in
\S\ref{S:construction}. 
By Lemma~\ref{L:eigenvector},
$\ker (D_\Delta^*-iI)\ne\{0\}$.
By Lemma~\ref{L:cEbounded},
the operator $\cE$ is bounded.
As $\Gamma=\Delta+\cE$,
it follows from Lemma~\ref{L:perturb}
that $\ker (D_\Gamma^*-iI)\ne\{0\}$.
Hence, by Lemma~\ref{L:ker}, 
there exist non-zero $\alpha,\beta\in\ell^2(\ZZ^+,\CC)$
such that \eqref{E:cond1'} and \eqref{E:cond2'} both hold.
\end{proof}

Finally, we deduce the main result:

\begin{proof}[Proof of Theorem~\ref{T:main}]
Combine Theorems~\ref{T:equiv} and \ref{T:soln}.
\end{proof}


\section*{AI tools declaration}

In preparing this work, the author used ChatGPT to develop and explore strategies for the solution of Problem~\ref{Pb:Anal1}. The final mathematical  text was written entirely by the author, who thoroughly corrected, verified, and streamlined the proposed strategy. The author accepts full responsibility for the accuracy and integrity of the final results presented in this paper.

\section*{Acknowledgements}

The author thanks the following people for valuable discussions on
the topic of this article: Eugenio Dellepiane, Soumitra Ghara, 
Bartosz Malman, Javad Mashreghi, Ryan O'Loughlin, Marcu-Antone Orsoni, Pierre-Olivier Paris\'e,
Annabelle Walsh and Malik Younsi.

\bibliographystyle{plain}
\bibliography{biblist}

\begin{thebibliography}{10}

\bibitem{dBR66a}
L.~de~Branges and J.~Rovnyak.
\newblock Canonical models in quantum scattering theory.
\newblock In {\em Perturbation {T}heory and its {A}pplications in {Q}uantum
  {M}echanics ({P}roc. {A}dv. {S}em. {M}ath. {R}es. {C}enter, {U}.{S}. {A}rmy,
  {T}heoret. {C}hem. {I}nst., {U}niv. of {W}isconsin, {M}adison, {W}is.,
  1965)}, pages 295--392. Wiley, New York-London-Sydney, 1966.

\bibitem{dBR66b}
L.~de~Branges and J.~Rovnyak.
\newblock {\em Square summable power series}.
\newblock Holt, Rinehart and Winston, New York-Toronto-London, 1966.

\bibitem{EFKMR16}
O.~El-Fallah, E.~Fricain, K.~Kellay, J.~Mashreghi, and T.~Ransford.
\newblock Constructive approximation in de {B}ranges-{R}ovnyak spaces.
\newblock {\em Constr. Approx.}, 44(2):269--281, 2016.

\bibitem{FM16}
E.~Fricain and J.~Mashreghi.
\newblock {\em The theory of {$\cal{H}$$(b)$} spaces. {V}ol. 2}, volume~21 of
  {\em New Mathematical Monographs}.
\newblock Cambridge University Press, Cambridge, 2016.

\bibitem{GMR24}
S.~Ghara, J.~Mashreghi, and T.~Ransford.
\newblock Summability and duality.
\newblock {\em Publ. Mat.}, 68(2):407--429, 2024.

\bibitem{MPR22a}
J.~Mashreghi, P.-O. Paris\'e, and T.~Ransford.
\newblock Failure of approximation of odd functions by odd polynomials.
\newblock {\em Constr. Approx.}, 56(1):35--43, 2022.

\bibitem{MPR22b}
J.~Mashreghi, P.-O. Paris\'e, and T.~Ransford.
\newblock Power-series summability methods in de {B}ranges--{R}ovnyak spaces.
\newblock {\em Integral Equations Operator Theory}, 94(2):Paper No. 20, 17,
  2022.

\bibitem{MR18}
J.~Mashreghi and T.~Ransford.
\newblock Outer functions and divergence in de {B}ranges--{R}ovnyak spaces.
\newblock {\em Complex Anal. Oper. Theory}, 12(4):987--995, 2018.

\bibitem{MR19}
J.~Mashreghi and T.~Ransford.
\newblock Linear polynomial approximation schemes in {B}anach holomorphic
  function spaces.
\newblock {\em Anal. Math. Phys.}, 9(2):899--905, 2019.

\bibitem{Ra26}
T.~Ransford.
\newblock On the coefficient formula for de {B}ranges--{R}ovnyak norms.
\newblock {\em Integral Equations Operator Theory}.
\newblock In press. Preprint arXiv:2605.30114.

\bibitem{Sa86}
D.~Sarason.
\newblock Doubly shift-invariant spaces in {$H^2$}.
\newblock {\em J. Operator Theory}, 16(1):75--97, 1986.

\bibitem{Sa94}
D.~Sarason.
\newblock {\em Sub-{H}ardy {H}ilbert spaces in the unit disk}, volume~10 of
  {\em University of Arkansas Lecture Notes in the Mathematical Sciences}.
\newblock John Wiley \& Sons, Inc., New York, 1994.
\newblock A Wiley-Interscience Publication.

\bibitem{Te14}
G.~Teschl.
\newblock {\em Mathematical methods in quantum mechanics}, volume 157 of {\em
  Graduate Studies in Mathematics}.
\newblock American Mathematical Society, Providence, RI, second edition, 2014.
\newblock With applications to Schr\"odinger operators.

\end{thebibliography}

\end{document}